%% file: RazumikhinBoundsARXIV4.tex
\documentclass[final,1p,times,twocolumn,letterpaper]{elsarticle}
\usepackage{amsmath,amssymb,amsfonts,amsthm} 
\usepackage[pdfstartview=FitH,colorlinks=true,bookmarks=false]{hyperref}
\usepackage{subfigure}
\usepackage{color}
\usepackage{titlesec}

\numberwithin{equation}{section}

\newtheorem{thm}{Theorem}[section]

\newtheorem{lem}[thm]{Lemma}
\newtheorem{assume}[thm]{Assumption}
\theoremstyle{definition}
\newtheorem{example}[thm]{Example}
\theoremstyle{definition}

\theoremstyle{definition}
\newtheorem{defn}[thm]{Definition}

\journal{arxiv}

\renewcommand{\leq}{\leqslant}
\renewcommand{\geq}{\geqslant}
\newcommand{\cC}{C}
\newcommand{\PC}{PC}
\newcommand{\R}{\mathbb{R}}
\newcommand{\cone}{\mathop \Sigma _{\Delta}}
\newcommand{\wedg}{\mathop \Sigma_{w}}
\newcommand{\cusp}{\mathop \Sigma_{c}}
\newcommand{\FigureCaption}{\caption}
\newcommand{\TableCaption}{\caption}
\renewcommand{\phi}{\varphi}
\renewcommand{\epsilon}{\varepsilon}

\begin{document}

\begin{frontmatter}

\title{Lyapunov-Razumikhin techniques for state-dependent delay differential equations}

\author[McGill]{A.R. Humphries}
\ead{tony.humphries@mcgill.ca}
\author[Queens]{F.M.G. Magpantay}
\ead{felicia.magpantay@queensu.ca}

\address[McGill]{Department of Mathematics and Statistics, McGill University, 805 Sherbrooke St. W., Montreal, QC, Canada H3A 0B9}
\address[Queens]{Department of Mathematics and Statistics, Queen's University, 48 University Avenue,
Kingston, ON, Canada K7L 3N6}

\begin{abstract}
We present Lyapunov stability and asymptotic stability theorems for steady state solutions of general state-dependent delay differential equations (DDEs) using Lyapunov-Razumikhin methods.
Our results apply to DDEs with multiple discrete state-dependent delays, which may be nonautonomous for the
Lyapunov stability result, but autonomous (or periodically forced) for the asymptotic stability result.
Our main technique is to replace the DDE by a nonautonomous ordinary differential equation (ODE) where the
delayed terms become source terms in the ODE. The asymptotic stability result and
its proof are entirely new, and based on a contradiction argument together with the Arzel\`a-Ascoli theorem.
This approach alleviates the need to construct auxiliary functions to ensure the asymptotic contraction,
which is a feature of all other
Lyapunov-Razumikhin asymptotic stability results of which we are aware.

We apply our results to a state-dependent
model equation which includes Hayes equation as a special case, to directly establish asymptotic stability
in parts of the stability domain along with lower bounds on the size of the basin of attraction.
\end{abstract}

\end{frontmatter}

\section{Introduction}
\label{Sec:Intro}
\input{Intro.tex}

\section{Lyapunov-Razumikhin techniques for state-dependent DDEs}
\label{Sec:Extension}
\input{Extension1.tex}

\section{Model Equation Properties}
\label{Sec:ModelEquation}
\input{ModelEquationProperties.tex}

\section{Asymptotic stability for $\dot u(t)=\mu u(t) +\sigma u(t-a-cu(t))$
using $\mathcal{E}_{(k)}(\delta,x)$}
\label{Sec:Stability:k=123}
\input{ModelEquation123.tex}

%
%


\section{Comparison of the stability regions}
\label{Sec:StabilityDDE:Measurements}
\input{Measurements.tex}

\section{Basins of attraction}
\label{Sec:Basins}
\input{Basins.tex}


\section{Conclusions}
\label{Conclusions}
\input{Conclusions.tex}

\section*{Acknowledgments}
ARH is grateful to Tibor Krisztin, John Mallet-Paret, Roger Nussbaum and Hans-Otto Walther
for productive discussions and suggestions, and to the National Science and Engineering Research Council (NSERC), Canada for funding through the Discovery Grant program. FMGM is grateful to Jianhong Wu for helpful discussions, and to McGill University, York University, the Institut des Sciences Math\'ematiques (Montreal, Canada) and NSERC for funding. We are grateful to an anonymous referee whose feedback significantly improved the manuscript.

\def\bibsection{\section*{References}}

\bibliographystyle{plain}


\appendix

\titleformat{\section}{\normalfont\bfseries}{Appendix \thesection.}{1em}{}
\renewcommand{\thesection}{\Alph{section}}
\renewcommand{\thefigure}{\arabic{figure}}


\section{An explicit expression for the region \boldmath$P(\delta,c,2)<\delta$}
\label{Appendix:Explicit}
\input{AppendixExplicit.tex}

\end{document}

%% file: Intro.tex
We consider the following general delay differential equation (DDE) in $d$ dimensions with $N$ discrete state-dependent delays,
\begin{equation}
\left\{ \begin{array}{ll}
   \dot u(t) = f\bigl( t,u(t),u(t-\tau_1(t,u(t))),\ldots, u(t-\tau_N(t,u(t))) \bigr), & t\geqslant t_0,  \\
   u(t) = \varphi(t), & t\leqslant t_0,
 \end{array}  \right. \label{Eq:GeneralMultipleDelay}
\end{equation}
and prove Lyapunov stability and asymptotic stability results using Lyapunov-Razumikhin techniques. We apply
our results to the model state-dependent DDE
\begin{equation}
\left\{ \begin{array}{ll}
\dot u(t)=\mu u(t)+\sigma u(t-a-cu(t)), & t\geq 0,\\
u(t) = \varphi(t), & t\leqslant 0,
\end{array} \right.
\label{Eq:1Delay}
\end{equation}
with $a>0$, as an example of \eqref{Eq:GeneralMultipleDelay}, to 
directly show asymptotic stability of the trivial solution in
parts of the stability domain, and derive  bounds on the basin of attraction. 

Differential equations with state-dependent delays arise in many applications including
milling \cite{IST07}, control theory \cite{W03}, haematopoiesis \cite{CHM16} and economics \cite{Mackey89}.
There is a well-established theory for retarded
functional differential equations (RFDEs) as infinite-dimensional dynamical systems on function spaces
\cite{DGVLW95,GW13,HaleLunel:1}, which encompasses problems with constant or prescribed delay,
but very little of this theory is directly applicable to state-dependent
delay problems. Extending the theory to state-dependent DDEs, including equations of the
form \eqref{Eq:GeneralMultipleDelay} is the subject of ongoing study.
See \cite{HartungKrisztinWaltherWu:1} for further examples and
a review of progress.

The model state-dependent DDE \eqref{Eq:1Delay} includes the constant delay DDE sometimes known as Hayes equation
\cite{Hayes50} as a special case when $c=0$.
Hayes equation is a standard model problem used to illustrate
stability theory for constant delay DDEs in most texts on the subject
including \cite{HaleLunel:1,IS11,Smith:1}, as well as being a standard numerical analysis test problem \cite{BellenZennaro:1}.
Hayes equation is also used to illustrate Lyapunov-Razumikhin stability results in
\cite{Barnea:1,HaleLunel:1,Myshkis:1}.

The state-dependent DDE \eqref{Eq:1Delay} was introduced by Mallet-Paret and Nussbaum, and is a natural
generalisation of Hayes equation to a state-dependent DDE with a single delay which is linearly state-dependent.
But whereas Hayes equation is linear, the DDE \eqref{Eq:1Delay} is nonlinear and can admit limit cycles.
Mallet-Paret and Nussbaum
investigate the existence and form of the slowly oscillating periodic solutions of a singularly
perturbed version of \eqref{Eq:1Delay} in detail in \cite{MalletParetNussbaum:4} and use it as an illustrative
example for more general problems in \cite{MalletParetNussbaum:1, MalletParetNussbaum:2, MalletParetNussbaum:3}.
This DDE is also studied in \cite{CHK:1,HBCHM:1,Humphries:1,KE:1,MH:1}.

To study the stability of steady states of RFDEs,
Krasovskii \cite{Krasovskii:1} extended
the method of Lyapunov functions for ODEs to Lyapunov functionals
for RFDEs.
Lyapunov theorems for 
RFDEs require the time derivative of the
functional 
to be nonpositive or strictly negative,
similar to the theorems for ODEs using Lyapunov functions \cite{HaleLunel:1, Krasovskii:1}. However,
finding functionals 
with this property for RFDEs is much harder than in the ODE
case. 
Razumikhin \cite{Razumikhin:1} developed the theory on how one might go from using the more difficult Lyapunov functionals 
back to Lyapunov functions again. His fundamental idea is that it is only necessary to require a constraint on the derivative of the Lyapunov function 
whenever  the solution is about to exit a ball centered at the steady state.
Following this approach,
Barnea \cite{Barnea:1} presents a Lyapunov stability theorem for RFDEs
and also considers Hayes equation.
A comprehensive discussion of Lyapunov functionals and functions for general RFDEs is
presented by Hale and Verduyn Lunel in chapter 5 of \cite{HaleLunel:1}.
Other works with Razumikhin-type results include
\cite{Hale:1,IvanovLizTrofimchuk:1,Kato:1,Krasovskii:1,Krisztin:1,Krisztin:3,Myshkis:1,Yorke:1}.
Of these \cite{Krisztin:3,Myshkis:1,Yorke:1} include results tailored for time-dependent delays.

Although state-dependent DDEs can be formulated as RFDEs, results tailored for RFDEs often do not apply directly to state-dependent DDEs. For example, Barnea \cite{Barnea:1} only establishes Lyapunov stability and to do so
assumes a Lipschitz condition on a Banach space of continuous functions, which is well-known not to hold
for general state-dependent DDEs of the form \eqref{Eq:GeneralMultipleDelay} \cite{HartungKrisztinWaltherWu:1}, and is
easily shown not to hold for \eqref{Eq:1Delay}. Other authors, such as
Hale and Verduyn Lunel \cite{HaleLunel:1}, make the weaker continuity assumptions,
but use auxiliary functions  to establish Lyapunov stability and uniform asymptotic stability.
The construction of such functions is nontrivial in all but the simplest examples,
and we have not seen them constructed for a state-dependent DDE.

Rather than try to circumvent the problems that arise with RFDEs,
in Section~\ref{Sec:Extension} we will develop new proofs of Lyapunov stability and asymptotic stability
for the state-dependent DDE \eqref{Eq:GeneralMultipleDelay} with $N$ discrete delays using
Lyapunov-Razumikhin techniques.
In Assumption~\ref{Assume:DDE} we state the assumptions that we make on the
nonautonomous DDE \eqref{Eq:GeneralMultipleDelay} with $N$ (state-dependent) delays, the main ones being that
$f$ is locally Lipschitz with respect to its arguments in $\R^d$ and the delays
are locally bounded near the steady state.
Then in Theorem~\ref{Thm:Lyapunov_NA} we provide sufficient conditions for Lyapunov
stability of a steady state of the DDE. The main idea behind the proof is
the conversion of the DDE into an auxiliary ODE problem where the delayed terms are regarded as source terms.
In Theorem~\ref{Thm:Technique} we establish asymptotic stability of the steady state when the
DDE \eqref{Eq:GeneralMultipleDelay} is autonomous. For simplicity of exposition we present the proof for the case
of a single delay, but the result remains true for multiple delays or periodically forced problems.
This result is significantly different to previous Lyapunov-Razumikhin asymptotic stability results
which require auxiliary functions to
establish uniform asymptotic stability. 
In contrast, Theorem~\ref{Thm:Technique} does not require the construction of any auxiliary functions,
and is proved
using the auxiliary ODE by a contradiction argument, which shows there cannot exist a solution which is
not asymptotic to the steady state.

Theorems~\ref{Thm:Lyapunov_NA} and~\ref{Thm:Technique} establish Lyapunov stability and asymptotic stability when the
solutions of the auxiliary ODE have certain properties, but to determine those properties exactly would require
the solutions of the DDE. So, in Section~\ref{Sec:Extension}, we also show  how to define a family of
ODE problems that are subject to constraints defined by bounds on the DDE solution and its derivatives,
which can be determined without solving the DDE.
Lemma~\ref{Lem:FiniteBoundI} establishes bounds on the growth of solutions to \eqref{Eq:GeneralMultipleDelay}
which are used to ensure solutions remain bounded for sufficiently long ($k$ times the largest delay for
some integer $k$) to acquire $k$ bounded derivatives.
Stability is then established from the solution properties of this family of constrained ODE problems.




In Section~\ref{Sec:ModelEquation} we review the stability region of the steady state of the model equation~\eqref{Eq:1Delay} which is known \cite{GyoriHartung:1} to be the same for the state-dependent ($c\ne0$) and constant delay cases ($c=0$). We also
consider the properties of the auxiliary ODE we define for this problem, and define sets and functions that are required in the following sections to apply our Lyapunov-Razumikhin results.

In Section~\ref{Sec:Stability:k=123}
we apply Theorem~\ref{Thm:Technique} to provide a proof of asymptotic stability of the steady state of~\eqref{Eq:1Delay}
in subsets of the known stability region, together with
lower bounds on the size of the basin of attraction.
This result is given as Theorem~\ref{Thm:StabilityDDE:Razumikhin_k=123}.
Since the delayed inputs to the auxiliary ODE have $k-1$ bounded derivatives, for the
$k=2$ and $k=3$ results we establish
suitable bounds on the first and second derivatives of the ODE source terms, while the
$k=1$ result, does not require any differentiability of these terms.

The expressions for the stability regions derived
in Sections~\ref{Sec:Stability:k=123}
all involve a term that needs to be maximized over a closed interval (see the definitions of the $P(\delta,c,k)$ functions
in Definition~\ref{Defn:StabilityDDE:etaIP}).
For $k=1$ 
this maximum is readily evaluated, while for $k=2$
an expression for the maximum is established in Theorem~\ref{Thm:StabilityDDE:simplify} (whose proof is
given in Appendix~\ref{Appendix:Explicit}).
Plots and measurements of the derived asymptotic stability regions in $(\mu,\sigma)$ parameter space are given in Section~\ref{Sec:StabilityDDE:Measurements}, where it is seen that these
regions grow with the integer $k$,
but do not appear to fill out the entire stability region in the case $\mu\ne0$.
In Section~\ref{Sec:StabilityDDE:Measurements},
we also briefly review previous work  on the
$c=0$ constant delay case of \eqref{Eq:1Delay} with $\mu=0$ and $\mu\ne0$
and point out an error in the results of Barnea \cite{Barnea:1}.

In Section~\ref{Sec:Basins} we present two examples of solutions which are not asymptotic to the steady state
of the model equation~\eqref{Eq:1Delay}
when $\mu\geq0>\sigma$, and which hence give upper bounds on the radius of the largest ball contained in the
basin of attraction of the steady state. We compare these with the
lower bounds on the basin of attraction given by
Theorem~\ref{Thm:StabilityDDE:Razumikhin_k=123} for $k=1$, $2$ and $3$.
In Section~\ref{Conclusions} we present brief conclusions, and compare and contrast our approach with linearization.

%% file: Extension1.tex
Here we state and prove our main theorems to establish the Lyapunov stability and asymptotic stability of
steady state solutions to state-dependent DDEs of the form \eqref{Eq:GeneralMultipleDelay}.

Let $\mathbb{R}^d$ be the $d$-dimensional
linear vector space over the real numbers equipped with the Euclidean inner product
$\cdot$ and Euclidean norm $|\cdot|$.
We denote by $B(0,\delta)$ the closed ball centred at zero with radius $\delta>0$ in $\mathbb{R}^d$.

\begin{defn} \label{Defn:r}
For any $\delta>0$ define
\begin{equation} \label{eq:r}
r(\delta)=\sup_{\begin{subarray}{c}t\geqslant t_0,|u|\leqslant \delta,\\i=1,...,N\end{subarray}} \tau_i(t,u).
\end{equation}
\end{defn}

Let $\delta_0>0$, and let $C=C\bigl([t_0-r(\delta_0),t_0],\mathbb{R}^d\bigr)$
be the Banach space of continuous functions mapping $[t_0-r(\delta_0),t_0]$
to $\mathbb{R}^d$ with the supremum norm denoted $\|\cdot\|$.
We will consider continuous initial functions $\phi\in C$,
and by a solution of \eqref{Eq:GeneralMultipleDelay} we mean a function $u\in C^1([t_0,t_f),\R^d)$ which satisfies
\eqref{Eq:GeneralMultipleDelay} for $t\in[t_0,t_f)$ with $t_f>t_0$ and $u(t_0)=\phi(t_0)$.
We will assume that \eqref{Eq:GeneralMultipleDelay} has a steady state at $u=0$ and make the following assumptions throughout.

\begin{assume}
\label{Assume:DDE}
Let $d$, $N$ and $k\in \mathbb{Z}$, $d\geqslant 1$, $N\geqslant 1$, $k\geqslant 1$ and $t_0\in\mathbb{R}$.
\begin{enumerate}
\item
$f:[t_0,\infty)\times \mathbb{R}^{(N+1)d} \to \mathbb{R}^d$ and $\tau_i:[t_0,\infty)\times \mathbb{R}^d \to \mathbb{R}$ for $i=1,\dotsc,N$ are continuous functions of their variables.
\item
$f(t,0,0,\dotsc,0)=0$ for all $t\geqslant t_0$.
\item
The constant $\tau_{\text{max}}=\max_{i=1,...,N}\sup_{t\geq t_0}\tau_i(t,0)$ satisfies $\tau_{\text{max}}\in(0,\infty)$.
\item
There exist positive constants $L_0$, $L_1,\dotsc,L_N$ and $\delta_{0}$ such that
$$\bigl|f(t,u,v_1,\dotsc,v_N)-f(t,\tilde u,\tilde v_1,\dotsc,\tilde v_N)\bigr|
\leqslant L_0|u-\tilde u| + L_1 |v_1-\tilde v_1|+\dots+L_N|v_N-\tilde v_N|$$
for all $t\geqslant t_0$ and $u,v_1,\dotsc,v_N,\tilde u, \tilde v_1,\dotsc,\tilde v_N \in
B(0,\delta_0)$. Let $L=L_0+L_1+\dots+L_N$.
\item
The delay terms $\tau_i(t,u)$ are nonnegative and Lipschitz continuous in $u$ on
$[t_0,\infty)\times B(0,\delta_0)$ with Lipschitz constants $L_{\tau_i}$.
\item
$f(t,u,v_1,\dotsc,v_N)$ is at least $\max\{k-2,0\}$ times differentiable in its $u$ and $v$ variables, and $\tau_i(t,u)$ is at least $\max\{k-1,0\}$ times differentiable in $u$.
\item
$\phi\in C:=C\bigl([t_0-r(\delta_0),t_0],\mathbb{R}^d\bigr)$.
\end{enumerate}
\end{assume}

Notice that Assumption~\ref{Assume:DDE} item 4 fixes our choice of $\delta_0>0$.
Now, for $\delta\in(0,\delta_0]$, Assumption~\ref{Assume:DDE}, items 3 and 5 imply $0<\tau_\text{max}\leq r(\delta)\leq \tau_{\max}+\delta\max_iL_{\tau_i}<\infty$
with $\lim_{\delta\to0}r(\delta)=\tau_\text{max}>0$.
In contrast, the bound on the delay terms used in RFDE theory by Hale and Verduyn Lunel \cite{HaleLunel:1} is a global bound on $\tau_i(t,u)$ over all $t$ and $u$.

The Lipschitz continuity conditions in Assumption~\ref{Assume:DDE}, items 4--5 are based on those
of Driver \cite{Driver:1}. From the results in \cite{Driver:1},
these conditions ensure local existence
and uniqueness of a solution to \eqref{Eq:GeneralMultipleDelay} close to the steady
state solution for all sufficiently small Lipschitz continuous initial functions $\phi$, and
existence of a solution for all sufficiently small continuous initial functions $\phi$.

We will prove Lyapunov stability and asymptotic stability by contradiction, by showing that
there is no
solution of \eqref{Eq:GeneralMultipleDelay} that does not have the required stability property.
Thus we do not need to make assumptions 
specifically to ensure existence or uniqueness
of solutions. Rather, we make sufficient assumptions so solutions that do exist have the properties
we require.

Since with state-dependent delays it can be hard ensure \emph{a priori} that the delays $\tau_j(t,u(t))$ are strictly positive along solutions, and problems with vanishing delays can be interesting, we do \emph{not} assume that delays $\tau_j(t,u)$ are strictly positive in Assumption~\ref{Assume:DDE}. This makes our results more widely applicable, allowing us to deal with state-dependent DDEs without having to first impose or prove that the delays do not vanish along solutions. The possibility of vanishing delays will on the other hand complicate some of our proofs.
Results similar to
Lemmas~\ref{Lem:FiniteBoundI} and~\ref{Lem:FiniteBoundII} below
are well known for time-varying and nonvanishing delays and proved using a Gronwall lemma and the method of steps; however the method of steps is not applicable with vanishing delays, and we will need a more technical proof.

The following lemmas provide bounds on the growth of solutions to \eqref{Eq:GeneralMultipleDelay}. For initial functions $\phi\in C$ we will use
Lemma~\ref{Lem:FiniteBoundI}
to ensure that solutions $u(t)$
remain bounded for a sufficiently long time interval so that $u(t)$ acquires the differentiability
that we require.

\begin{lem}
\label{Lem:FiniteBoundI}
Suppose that Assumption~\ref{Assume:DDE} is satisfied.
Let $T>t_0$ and $\delta\in(0,\delta_0]$. Let $|\varphi(t)|\leqslant \delta e^{L(t_0-T)}$ for $t\leqslant t_0$. Then any solution of \eqref{Eq:GeneralMultipleDelay}
defined for $t\in[t_0,T]$ satisfies
$$|u(t)|\leqslant \delta e^{L(t-T)}\leqslant \delta
\quad \forall t\in[t_0,T].$$
\end{lem}

\begin{proof}
Let $w(t)=\sup_{s\leq t}|u(s)|$, then for $t\in[t_0,T]$
\begin{align*}
\dot w(t)
& = \left\{\begin{array}{ll}
0 & \text{if }|u(t)|<w(t)\text{ or }\frac{d}{dt}|u(t)|<0,\\
\displaystyle\frac{d}{dt}|u(t)| & \text{otherwise,}
\end{array}\right. \\
& \leq |\dot u(t)|
=\bigl|f\bigl(t,u(t),u(t-\tau_1(t,u(t))),\dotsc,u(t-\tau_N(t,u(t)))\bigr)\bigr|.
\end{align*}
But
Assumption~\ref{Assume:DDE}, items 2 and 4 imply that
$$\bigl|f(t,u,v_1,\dotsc,v_N)\bigr|\leqslant L_0|u| + L_1 |v_1|+\dots+L_N|v_N|,$$
while
$t-\tau_j(t,u(t))\leqslant t$ implies $|u(t-\tau_j(t,u(t)))|\leqslant w(t)$. Hence
\begin{align*}
\dot w(t) &
\leqslant L_0 |u(t)|+L_1|u(t-\tau_1(t,u(t)))|+\dotsc+ L_N|u(t-\tau_N(t,u(t)))| \\
& \leqslant (L_0+L_1+\dotsc +L_N)w(t)=L w(t).
\end{align*}
The result follows easily.
\end{proof}

Lemma~\ref{Lem:FiniteBoundII} below, which gives a bound on that rate that nearby trajectories diverge from one another,
will be essential in the proof of asymptotic stability
in Theorem~\ref{Thm:Technique}. Winston's example \cite{Winston74} of nonuniqueness of solutions for a continuous but non-Lipschitz initial function $\phi$, would provide a counter-example to the lemma if the Lipschitz continuity requirement on the initial function $\phi$ were relaxed. Although we require $\phi$ to be Lipschitz in Lemma~\ref{Lem:FiniteBoundII}, when we consider stability
for the state-dependent DDE \eqref{Eq:GeneralMultipleDelay} we will not impose a Lipschitz condition on $\phi$, but instead apply Lemma~\ref{Lem:FiniteBoundI} to ensure that the solution remains bounded for sufficiently long to acquire the required smoothness.

\begin{lem}
\label{Lem:FiniteBoundII}
Suppose that Assumption~\ref{Assume:DDE} is satisfied.
Let $T>t_0$ and $\epsilon>0$.
Let $\delta\in(0,\delta_0]$
and
$L_{\delta}=(1+\delta\sum_{j=1}^NL_jL_{\tau_j})L$.
Suppose that $u(t)$ and $\tilde u(t)$ are solutions to \eqref{Eq:GeneralMultipleDelay} with initial functions $\varphi$ and $\tilde\varphi\in C$ respectively,
where both $|\varphi(t)|\leq\delta$
and $|\tilde\varphi(t)|\leq\delta$ for $t\leq t_0$,
and $|u(t)|\leq\delta$ and $|\tilde u(t)|\leq\delta$ for all $t\in[t_0,T]$.
Let $\phi$ be Lipschitz with Lipschitz constant $L\delta$, and suppose
$|\varphi(t)-\tilde\varphi(t)|\leqslant\epsilon$ for $t\leqslant t_0$.
Then
\begin{equation} \label{eq:IIbd}
|u(t)-\tilde u(t)|\leqslant \epsilon e^{L_{\delta}(t-t_0)} \quad\forall t\in[t_0,T].
\end{equation}
\end{lem}

\begin{proof}
The proof is quite similar to the proof of the previous lemma, but care must be taken with the state-dependent delays. Since
$|u(t)|\leq\delta$ for all $t\in[t_0,T]$ it follows from
Assumption~\ref{Assume:DDE}, items 2 and 4 that
$$|\dot{u}(t)|  \leq L_0|u(t)| +\sum_{j=1}^N L_j|u(t-\tau_j(t,u(t)))|
\leq \delta \sum_{j=0}^N L_j =L\delta.$$
Since $u(t)=\phi(t)$ for $t<0$, and by assumption $\phi(t)$ also has Lipschitz constant $L\delta$,
it follows for all $t_1\leq T$ and all $t_2\leq T$ that
\begin{equation} \label{eq:timebd}
|u(t_1)-u(t_2)|\leq L\delta|t_1-t_2|.
\end{equation}

Let $w(t)=\sup_{s\leq t}|u(s)-\tilde u(s)|$, then for $t\in[t_0,T]$,
similar to the proof of Lemma~\ref{Lem:FiniteBoundI}, we find
\begin{align*}
\dot w(t) & \leq |\dot u(t)-\dot{\tilde {u}}(t)| \\
&=\bigl|f\bigl(t,u(t),u(t-\tau_1(t,u(t))),\dotsc,u(t-\tau_N(t,u(t)))\bigr)
\\ &\qquad \qquad
-f\bigl(t,\tilde u(t),\tilde u(t-\tau_1(t,\tilde u(t))),\dotsc,\tilde u(t-\tau_N(t,\tilde u(t)))\bigr) \bigr| \\
& \leqslant L_0 |u(t)-\tilde u(t)|+
\sum_{j=1}^N L_j|u(t-\tau_j(t,u(t))) - \tilde u(t-\tau_j(t,\tilde u(t)))| \\
&\leqslant L_0 w(t)
+\sum_{j=1}^N L_j|u(t-\tau_j(t,\tilde u(t))) - \tilde u(t-\tau_j(t,\tilde u(t)))|\\
&\qquad \qquad +\sum_{j=1}^N L_j|u(t-\tau_j(t,u(t))) - u(t-\tau_j(t,\tilde u(t)))|.
\end{align*}

Since $t-\tau_j(t,\tilde u(t))\leqslant t$ it follows that
$|u(t-\tau_j(t,\tilde u(t)))-\tilde u(t-\tau_j(t,\tilde u(t)))|\leqslant w(t)$.
Using this and \eqref{eq:timebd} we see that
\begin{align*}
\dot w(t) & \leqslant \sum_{j=0}^N L_j w(t)
+L\delta\sum_{j=1}^N L_j|\tau_j(t,u(t))-\tau_j(t,\tilde u(t))|
 \leqslant L w(t)
+L\delta\sum_{j=1}^N L_jL_{\tau_j}|u(t)-\tilde u(t)|\\
& \leqslant\Bigl(1+\delta\sum_{j=1}^N L_jL_{\tau_j}\Bigr)L w(t)
= L_{\delta}w(t),
\end{align*}
from which
the result follows.
\end{proof}

In this section we will develop a constructive technique for showing Lyapunov stability and asymptotic stability based on Lyapunov-Razumikhin ideas. To establish Lyapunov stability in Theorem~\ref{Thm:Lyapunov_NA} we will show that solutions remain in a closed ball $B(0,\delta)$ of radius $\delta$ about the steady state. We are thus implicitly using the Lyapunov function $V(u)=|u|^2/2$, though it will not appear directly in our results. Asymptotic stability is established in Theorem~\ref{Thm:Technique}, by showing that all solutions that remain in $B(0,\delta)$ must converge to the steady state.
The application of these results to the model problem \eqref{Eq:1Delay} is demonstrated
in Sections~\ref{Sec:ModelEquation}--\ref{Sec:Stability:k=123}.

We will prove Lyapunov stability (and later asymptotic stability) by contradiction. Suppose the
DDE \eqref{Eq:GeneralMultipleDelay} has a solution which escapes the ball $B(0,\delta)$ at time $t^*$, by which we mean
\begin{equation} \label{eq:tstar}
t^*=\inf\bigl\{t: |u(t)|>\delta\bigr\}.
\end{equation}
For such a solution $u(t^*)=x$ with $|x|=\delta$. Now letting $v(\theta)=u(t^*+\theta)$ and
$$\eta_i(\theta)=u(t^*+\theta-\tau_i(t^*+\theta,u(t^*+\theta))), \qquad \theta\in[-r(\delta),0], \quad i=1,\ldots,N$$
for this solution $u(t)$,
we can rewrite the DDE \eqref{Eq:GeneralMultipleDelay} as a nonautonomous ODE
$$\dot{v}(\theta) = f(t^*+\theta,v(\theta), \eta_1(\theta),\dotsc, \eta_N(\theta))$$
where
$$v(-\tau_i(t^*,x))= u(t^*-\tau_i(t^*,x))=\eta_i(0), \quad i=1,\ldots,N.$$
Hence for each $i$ such that $\tau_i(t^*,x)>0$ the escaping solution of the DDE \eqref{Eq:GeneralMultipleDelay}
corresponds to a solution of an ODE boundary value problem (BVP) with $v(-\tau_i(t^*,x))=\eta_i(0)$ and $v(0)=x$.
To establish stability it is sufficient to show that the ODE BVP does not have any such solutions.
We will use Lemma~\ref{Lem:FiniteBoundI} to ensure that $|u(t)|\leq\delta$ for $t\leq t_0+kr(\delta)$
so that the forcing functions $\eta_i(\theta)$ acquire the regularity we require. Hence we define the set of
forcing functions $\eta_i$ which could correspond to an escaping trajectory as follows.

\begin{defn}\label{Defn:E_NA}
Suppose that Assumption~\ref{Assume:DDE} is satisfied for \eqref{Eq:GeneralMultipleDelay} and $k\geqslant 1$. Let $\delta\in(0,\delta_0]$, $|x|=\delta$ and $t^*\geqslant t_0+kr(\delta)$.
Define the set,
\begin{equation} \label{eq:Ekdxt}
E_{(k)}(\delta,x,t^*)=
\left\{
\begin{array}{l}
  (\eta_1,\dotsc,\eta_N):  \eta_i \in C^{k-1}\big([-r(\delta),0], B(0,\delta) \big), \\
\mbox{}\quad \text{and conditions 1 and 2 are satisfied.}\\
\end{array}\right\}.
\end{equation}
\begin{enumerate}
\item $x \cdot f(t^*,x,\eta_1(0),\dotsc,\eta_N(0)) \geqslant 0$,
\item For some initial function $\varphi\in\cC$, equation \eqref{Eq:GeneralMultipleDelay} has solution $u(t)$ such that
$\eta_i(\theta) = u(t^*+\theta-\tau_i(t^*,u(t^*+\theta))$
for $\theta\in[-r(\delta),0]$ for each $i\in\{1,\dotsc,N\}$.
\end{enumerate}
\end{defn}

Condition (1) in the definition is equivalent to $\tfrac{d}{dt}|u(t)|_{t=t^*}\geq0$, a necessary condition for the solution
to escape the ball $B(0,\delta)$ at time $t^*$.
In the following theorem we prove the first of our main results,
that if certain conditions hold for all functions in the sets $E_{(k)}(\delta,x,t^*)$ then the steady state of
\eqref{Eq:ScalarOneDelay} is Lyapunov stable. The condition \eqref{odecond_NA} implies that the ODE BVP
discussed above has no solutions, while \eqref{odecond2_NA} allows solutions with
$\tfrac{d}{d\theta}|v(\theta)|_{\theta=0}\leq0$.

The sets $E_{(k)}(\delta,x,t^*)$ however cannot be determined without solving \eqref{Eq:GeneralMultipleDelay},
so it is not practical to actually solve for them.
Instead, the conditions of the theorems can be shown to hold for larger sets that contain $E_{(k)}(\delta,x,t^*)$.
We prove the theorems first and then consider such larger sets.

\begin{thm}
\label{Thm:Lyapunov_NA}
Suppose that Assumption~\ref{Assume:DDE} is satisfied for \eqref{Eq:GeneralMultipleDelay}.
For $\delta\in(0,\delta_0]$,
$x\in\mathbb{R}^d$, $|x|=\delta$, define $E_{(k)}(\delta,x,t^*)$ as in Definition~\ref{Defn:E_NA}.
Consider the family of auxiliary ODE problems,
\begin{equation}\left\{
\begin{array}{ll}
\dot{v}(\theta) = f(t^*+\theta,v(\theta), \eta_1(\theta),\dotsc, \eta_N(\theta)), & \theta \in [-\tau_i(t^*,x),0],\\
v(-\tau_i(t^*,x))= \eta_i(0)\in\R^d, &
\end{array}\right.
\label{Eq:AuxiliaryODE_NA}
\end{equation}
for $i=1,\dotsc,N$ and $t^*\geqslant t_0+kr(\delta)$.
We denote the solution of \eqref{Eq:AuxiliaryODE_NA} by $v(x,\eta_1,\dotsc,\eta_N)(\theta)$
if we want to emphasize the dependence on $x$ and $\eta_i$, or just $v(\theta)$ otherwise.
Suppose there exists $\delta_1\in(0,\delta_0]$ such that for all $\delta\in(0,\delta_1)$,
and for every $x$ such that $|x|=\delta$ and all $t^*\geqslant t_0+kr(\delta)$, for all $(\eta_1,\dotsc,\eta_N)\in E_{(k)}(\delta,x,t^*)$ the solution
of \eqref{Eq:AuxiliaryODE_NA} for some $I\in\{1,\dotsc,N\}$ satisfies $\tau_I(t^*,x)>0$ and either
\begin{equation} \label{odecond_NA}
\tfrac{1}{\delta} v(x,\eta_1,\dotsc,\eta_N)(0)\cdot x< \delta,
\end{equation}
or
\begin{equation} \label{odecond2_NA}
\tfrac{1}{\delta} v(x,\eta_1,\dotsc,\eta_N)(0)\cdot x= \delta, \quad\text{and}\quad \dot{v}(x,\eta_1,\dotsc,\eta_N)(0)\cdot x \leq 0,
\end{equation}
then the zero solution to \eqref{Eq:GeneralMultipleDelay} is Lyapunov stable.
Moreover if $\delta\in(0,\delta_1)$ and
$|\varphi(s)|<\delta e^{-Lkr(\delta)}$ for $s\in[t_0-r(\delta),t_0]$ then
the solution of \eqref{Eq:GeneralMultipleDelay} satisfies $|u(t)|\leqslant \delta$ for all $t\geqslant t_0$.
\end{thm}

\begin{proof}
Let the hypothesis of the theorem hold and let $\delta\in(0,\delta_1)$.
By Lemma~\ref{Lem:FiniteBoundI},
the solution of \eqref{Eq:GeneralMultipleDelay}
satisfies $u(t)\in B(0,\delta)$ for all $t\leq t_0+ kr(\delta)$. We will prove that $u(t)\in B(0,\delta)$ for all $t\geq t_0$ by contradiction.

Assume that a solution escapes the ball $B(0,\delta)$ for the first time at some time $t^*\geq t_0+kr(\delta)$, where $t^*$ satisfies \eqref{eq:tstar}. 
Let $w(t)=\sup_{s\leq t}|u(s)|$, then $w(t^*)=|u(t^*)|$ and the fundamental theorem of calculus can be used to show that there exists a monotonically decreasing sequence $t_j$ with $t_j>t^*$, $w(t_j)=|u(t_j)|>\delta$
and $\lim_{j\to\infty}t_j=t^*$, $\lim_{j\to\infty}w(t_j)=\delta$ and
$\frac{d}{dt}|u(t_j)|>0$ for each $j$.

Assuming
$|\varphi(s)|<\delta e^{-Lkr(\delta)}$ for $s\in[t_0-r(\delta),t_0]$, there exists
$\lambda\in(0,1)$ such that
$|\varphi(s)|\leqslant \lambda \delta e^{-L kr(\delta)}$ for $s\in [t_0-r(\delta),t_0]$.
It follows from Assumption~\ref{Assume:DDE} and equation~\eqref{eq:r} that $r(\delta)$ is a continuous function of $\delta$ and hence
 $\delta e^{-Lkr(\delta)}$ is a continuous function of $\delta$ as well.
Thus for all $j$ sufficiently large $\lambda \delta e^{-Lkr(\delta)}<w(t_j)e^{-Lkr(w(t_j))}$.
Let $t^{**}=t_j$ for such a $j$ then
the solution escapes the ball $B(0,\delta^{**})$ for the first time at $t^{**}>t^*\geq t_0+kr(\delta)$.
Let $x=u(t^{**})$ so $|x|=\delta^{**}$ and
\begin{align}
0 < \tfrac12\tfrac{d}{dt}|u(t^{**})|^2 &= u(t^{**}) \cdot \dot{u}(t^{**}) ,\notag \\
&= u(t^{**}) \cdot f\big(t^{**},u(t^{**}),u(t^{**}-\tau_1(t^{**},u(t^{**}))),\dotsc,u(t^{**}-\tau_N(t^{**},u(t^{**})))\big), \notag \\
&=x\cdot f\big(t^{**},x,u(t^{**}-\tau_1(t^{**},u(t^{**}))),\dotsc,u(t^{**}-\tau_N(t^{**},u(t^{**})))\big).
\label{Eq:Technique1_NA}
\end{align}

\begin{figure}[t]
\begin{center}
\includegraphics[scale=0.55]{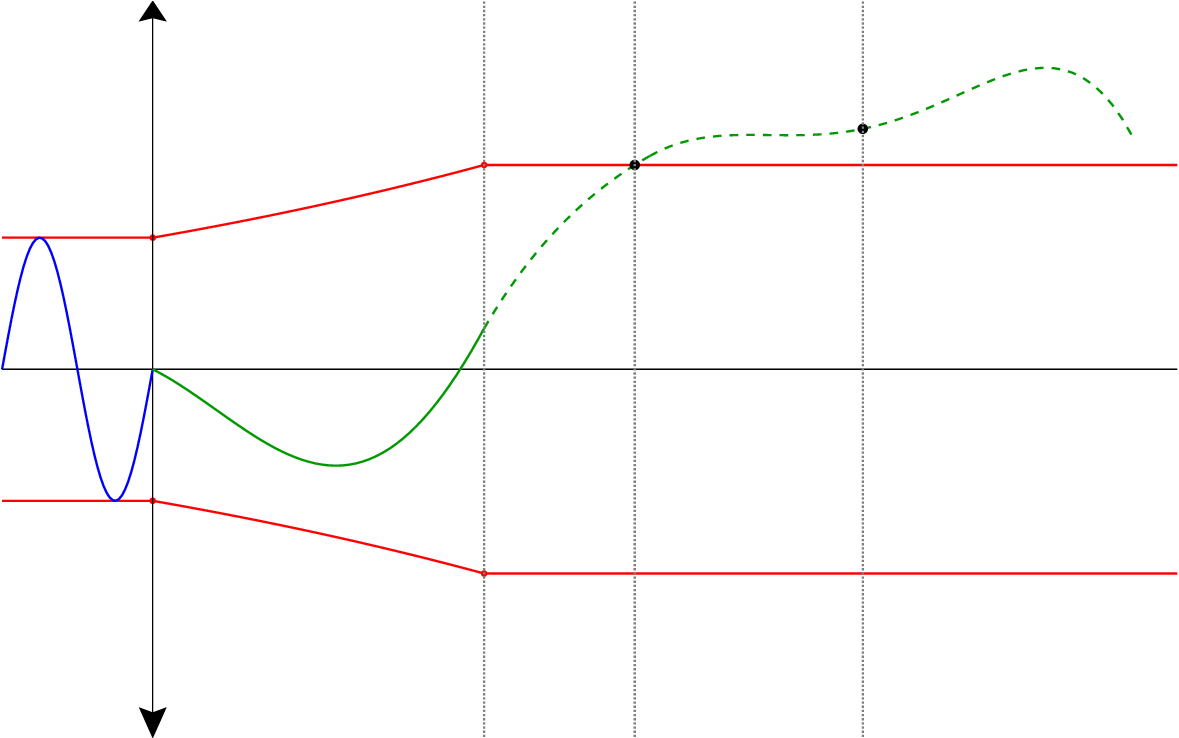}
\put(-270,100){{\footnotesize $t_0$}}
\put(-182,155){{\footnotesize\color{red} $\delta$}}
\put(-185,35){{\footnotesize\color{red} $-\delta$}}
\put(-295,140){{\footnotesize\color{red} $\delta e^{-Lkr}$}}
\put(-305,50){{\footnotesize\color{red} $-\delta e^{-Lkr}$}}
\put(-240,150){{\footnotesize\color{red} $\delta e^{L(t-t_0-kr)}$}}
\put(-245,40){{\footnotesize\color{red} $-\delta e^{L(t-t_0-kr)}$}}
\put(-295,120){{\footnotesize\color{blue} $\phi(t)$}}
\put(-182,100){{\footnotesize $t_0+kr$}}
\put(-142,100){{\footnotesize $t^*$}}
\put(-81,100){{\footnotesize $t^{**}$}}
\put(-143,143){{\footnotesize $|u(t^{*})|=\delta$}}
\put(-80,155){{\footnotesize $|u(t^{**})|=\delta^{**}$}}
\caption{\label{Fig:RazProof2_NA} Illustration of the proof of Theorem~\ref{Thm:Lyapunov_NA} in one dimension.}
\end{center}
\end{figure}
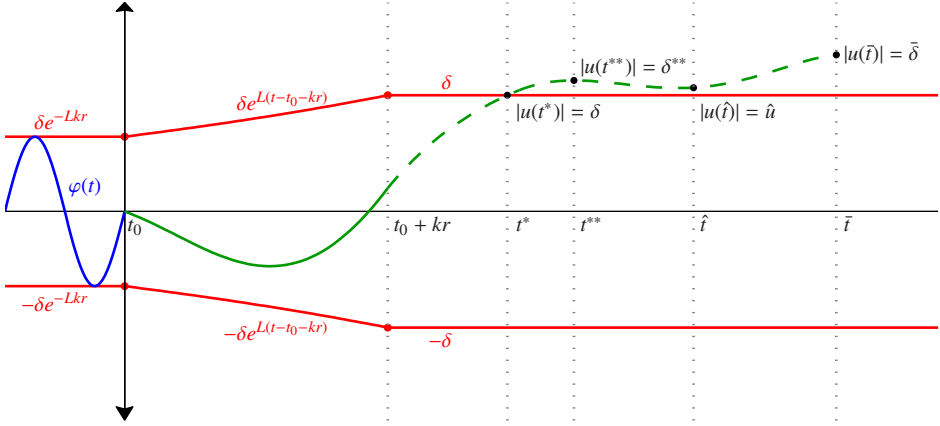

Since $|\varphi(s)|\leq \lambda \delta e^{-Lkr(\delta)}<\delta^{**} e^{-Lkr(\delta^{**})}$, it follows from Lemma~\ref{Lem:FiniteBoundI} that $t^{**} \geq t_0+kr(\delta^{**})$.
Now consider the auxiliary ODE problem \eqref{Eq:AuxiliaryODE_NA} for $I\in\{1,\ldots,N\}$ such that
either \eqref{odecond_NA} or \eqref{odecond2_NA} holds.
Let $v(\theta)=u(t^{**}+\theta)$, and noting that $[-\tau_I(t^{**},x ),0]\subseteq [-r(\delta^{**}),0]$,  let
$\eta_i(\theta) =u\bigl(t^{**}+\theta-\tau_i(t^{**}+\theta,u(t^{**}+\theta))\bigr)$
for $\theta\in[-r(\delta^{**}),0]$.
Then
$$t^{**}+\theta-\tau_i(t^{**}+\theta,u(t^{**}+\theta)) \geq t_0 +kr(\delta^{**})-2r(\delta^{**}) = t_0+(k-2)r(\delta^{**}),\quad \text{for }\theta\in[-r(\delta^{**}),0].$$
It follows that $\eta_i \in \cC^{k-1}([-r(\delta^{**}),0],B(0,\delta^{**}))$ for $i=1,\dotsc,N$.
From \eqref{Eq:Technique1_NA} we deduce that
\begin{equation}\label{lyapip_NA}
x\cdot f(t^{**},x, \eta_1(0),\dotsc,\eta_N(0)) > 0.
\end{equation}
Furthermore, we can consider
$$\eta_i(\theta) = U(t_0+kr(\delta^{**})+\theta - \tau_i(t_0+kr(\delta^{**})+\theta, U(t_0+kr(\delta^{**})+\theta)  ),
\quad \text{for $\theta\in[-r(\delta^{**}),0]$},
$$
where $U$ is the solution to \eqref{Eq:GeneralMultipleDelay} with initial function $\Phi$ defined as
\[
\Phi(s) = \left\{
\begin{array}{ll}
\varphi(s-kr(\delta^{**})+t^{**}), & s\leqslant t_0+ kr(\delta^{**})-t^{**},\\
u(s-kr(\delta^{**})+t^{**}), & s\in[ t_0 + kr(\delta^{**})-t^{**},t_0].
\end{array}\right.
\]
Thus, $(\eta_1,\dotsc,\eta_N)\in E_{(k)}(\delta^{**},x,t^{**})$.
Moreover,
\begin{align*}
&\dot v(\theta)=\dot u(t^{**}+\theta),\\
&= f\bigl(t^{**} +\theta,u(t^{**}+\theta),u(t^{**}+\theta-\tau_1(t^{**}+\theta,u(t^{**}+\theta)) ),
\dotsc, u(t^{**}+\theta-\tau_N(t^{**}+\theta,u(t^{**}+\theta)))\bigr),\\
&= f\bigl(t^{**} +\theta,v(\theta),\eta_1(\theta),\dotsc,\eta_N(\theta)\bigr).
\end{align*}
Also,
\[v(-\tau_I(t^{**},x))=u\bigl(t^{**} -\tau_I(t^{**},x)\bigr)=u\bigl(t^{**}-\tau_I(t^{**},u(t^{**}))\bigr)=\eta_I(0).\]
Thus $v$ is a solution to the ODE system \eqref{Eq:AuxiliaryODE_NA} with $i=I$ and $(\eta_1,\dotsc,\eta_N)\in E_{(k)}(\delta^{**},x,t^{**})$. But $v(0)=u(t^{**})=x$ with $|x|=\delta^{**}$
so $\frac{1}{\delta^{**}}v(x,\eta_1,\dotsc,\eta_N)(0)\cdot x=\delta^{**}$
which contradicts \eqref{odecond_NA}. Meanwhile, equation \eqref{lyapip_NA} implies that
$ \dot{v}(x,\eta_1,\dotsc,\eta_N)(0) \cdot x>0$ which contradicts \eqref{odecond2_NA}. 

Thus any solution with $|\varphi(s)|<\delta e^{-Lkr(\delta)}$ which escapes the ball
$B(0,\delta^{**})$ violates the conditions of the theorem, so $u(t)\in B(0,\delta^{**})$
for all $t\geq t_0$, and the zero solution to \eqref{Eq:GeneralMultipleDelay} is Lyapunov stable.
Moreover, since this is true with $t^{**}=t_j$ for all $j$ sufficiently large it follows that
$u(t)\in B(0,\delta)$ for all $t\geq t_0$.
\end{proof}


The proof of Theorem~\ref{Thm:Lyapunov_NA} is complicated by the auxiliary ODE \eqref{Eq:AuxiliaryODE_NA}
being nonautonomous. The solution of a nonautonomous ODE escaping the ball $B(0,\delta)$ for the first time at
$t^*$ neither implies that $\tfrac{d}{dt}|u(t^*)|>0$ nor that there exists $t^{**}>t^*$ such that
$|u(t)|>\delta$ for all $t\in(t^*,t^{**})$. As an illustration of this, consider the function $y(t)=\delta+t^3\sin^2(2\pi/t)$
which is easily seen to be continuously differentiable and crosses $\delta$ at $t=0$ with $y'(0)=0$, and for which
there does not exist any $\varepsilon>0$ such that $y(t)>\delta$ for all $t\in(0,\varepsilon)$.

Our second main result is to show asymptotic stability of the steady state $u=0$ if the auxiliary ODE
\eqref{Eq:AuxiliaryODE_NA} satisfies the strict inequality \eqref{odecond2_NA}. We will do this for autonomous
DDEs, and for simplicity of notation we only present the derivation for problems with one delay term ($N=1$).
The extension to multiple delays is straightforward, and we discuss the extension to periodically forced
nonautonomous DDEs after Theorem~\ref{Thm:Technique}. Hence we consider
autonomous DDEs of the form \eqref{Eq:GeneralMultipleDelay} with $N=1$ for which $f(t,u,v)=f(0,u,v)$ and
$\tau_1(t,u)=\tau(0,u)$ for all $t$. In this case we may set $t_0=0$, so
$\cC=\cC([-r(\delta),0],\R^d)$,
and rewrite \eqref{Eq:GeneralMultipleDelay} as
\begin{equation}
\left\{ {\begin{array}{ll}
   {\dot u(t) 
   =f\big(0,u(t),u(t-\tau(0,u(t)))\big),} & {t\geqslant 0,}  \\
   {u(t) = \varphi(t),} & t\leq0.
 \end{array} } \right. \label{Eq:ScalarOneDelay}
\end{equation}
By Assumption~\ref{Assume:DDE} item 2 the DDE \eqref{Eq:ScalarOneDelay} has the trivial steady state solution $u=0$.
We write the solution
to \eqref{Eq:ScalarOneDelay} as $u(\varphi)(t)$ when we want to emphasize initial conditions, or just $u(t)$ otherwise.
For equation \eqref{Eq:ScalarOneDelay} the auxiliary ODE introduced in \eqref{Eq:AuxiliaryODE_NA} becomes
\begin{equation}\left\{
\begin{array}{ll}
\dot v(\theta) = f\bigl( 0,v(\theta),\eta(\theta)\bigr), & \theta \in[-\tau(0,x),0],\\
v(-\tau(0,x))= \eta(0), &
\end{array}\right.
\label{Eq:AuxiliaryODE}
\end{equation}
and since we consider a single delay, there is one such auxiliary ODE associated with \eqref{Eq:ScalarOneDelay}.
We write the solution of \eqref{Eq:AuxiliaryODE}
as $v(x,\eta)(\theta)$ if we want to emphasize the dependence on $x$ and $\eta$, or just $v(\theta)$
otherwise. The sets $E_{(k)}(\delta,x,t)$ defined in \eqref{eq:Ekdxt} are no longer dependent on $t$ for
an autonomous DDE, and so we denote them by $E_{(k)}(\delta,x)$ which for \eqref{Eq:ScalarOneDelay} is defined as follows.

\begin{defn}\label{Defn:E}
Suppose that Assumption~\ref{Assume:DDE} is satisfied for \eqref{Eq:ScalarOneDelay} and $k\geqslant 1$. Let $\delta\in(0,\delta_0]$ and $|x|=\delta$.
Define the set
\begin{equation} \label{eq:Ekdx}
E_{(k)}(\delta,x)=
\left\{
\begin{array}{c}
  \eta:  \eta \in C^{k-1}\big([-\tau(0,x),0], B(0,\delta) \big), \text{ such that }
x \cdot f(0,x,\eta(0)) \geqslant 0,\\
\text{and for some initial function $\varphi\in\cC$ the solution $u(t)$ of \eqref{Eq:ScalarOneDelay} satisfies}\\
\eta(\theta) = u(kr(\delta)+\theta-\tau(0,u(kr(\delta)+\theta)) \text{ for } \theta\in[-\tau(0,x),0]
\end{array}\right\}.
\end{equation}
\end{defn}

To show asymptotic stability of the zero solution to \eqref{Eq:ScalarOneDelay} it is sufficient to strengthen
the conditions of Theorem~\ref{Thm:Lyapunov_NA} by requiring that solutions
of the auxiliary ODE problem 
satisfy the strict inequality \eqref{odecond_NA} and not the weaker condition \eqref{odecond2_NA}.
For the DDE \eqref{Eq:ScalarOneDelay} the condition \eqref{odecond_NA} becomes
\begin{equation} \label{odecond}
\frac{1}{\delta} v(x,\eta)(0)\cdot x< \delta,
\end{equation}
and we now show asymptotic stability when all solutions of the auxiliary ODE satisfy \eqref{odecond}. 

\begin{thm}[Asymptotic stability] \label{Thm:Technique}
Suppose that Assumption~\ref{Assume:DDE} is satisfied for \eqref{Eq:ScalarOneDelay}.
%
%
For $\delta\in(0,\delta_0]$,
$x\in\mathbb{R}^d$, $|x|=\delta$, define $E_{(k)}(\delta,x)$ as in Definition~\ref{Defn:E}.
If there exists $\delta_1\in(0,\delta_0]$ such that for all $\delta\in(0,\delta_1)$,
and for every $x$ such that $|x|=\delta$, for all $\eta\in E_{(k)}(\delta,x)$ the solution $v(x,\eta)(\theta)$
of the auxiliary ODE problem \eqref{Eq:AuxiliaryODE} satisfies \eqref{odecond}
then the results of Theorem~\ref{Thm:Lyapunov_NA} hold
and moreover, the zero solution of \eqref{Eq:ScalarOneDelay} is asymptotically stable.
Furthermore, if $|\varphi(s)|<\delta_1 e^{-Lkr(\delta_1)}$ for $s\in[-r(\delta_1),0]$ then $u(t)\to 0$ as $t\to\infty$.
\end{thm}

\begin{proof}
The only differences between the conditions of Theorem~\ref{Thm:Lyapunov_NA} and
Theorem~\ref{Thm:Technique} is that Theorem~\ref{Thm:Lyapunov_NA} allows a finite number of delays
and nonautonomous $f$ and requires the solution of the
auxiliary ODE problem \eqref{Eq:AuxiliaryODE_NA} to satisfy \eqref{odecond_NA} or \eqref{odecond2_NA},
while Theorem~\ref{Thm:Technique} assumes autonomous $f$, one delay, and requires that the strict inequality \eqref{odecond} hold. Thus it trivially follows that the requirements of Theorem~\ref{Thm:Lyapunov_NA} are satisfied,
and the results of Theorem~\ref{Thm:Lyapunov_NA} hold.

Let $\delta\in(0,\delta_1)$, $r=r(\delta)$ and $|\varphi(s)|< \delta e^{-Lkr}$ for $s\in[-r,0]$. Then
by Theorem~\ref{Thm:Lyapunov_NA} we have $|u(t)|\leqslant\delta$ for all $t\geqslant 0$. Consider such a solution.
Since $|u(t)|\leqslant\delta$ for all $t\geqslant 0$ we have
$\limsup _{t\to\infty}|u(t)| = \delta_\infty$ with $\delta_\infty\in[0,\delta]$,
and it remains only to show that $\delta_\infty=0$.

Since $\limsup _{t\to\infty}|u(t)|= \delta_\infty$ and $\{u:|u|=\delta_\infty\}$ is compact in $\R^d$
there exists $t_i$ such that $\lim_{i\to\infty}t_i=\infty$ and
$\lim_{i\to\infty}u(t_i)=x_\infty$
with $|x_\infty|=\delta_\infty$. Assume without loss of generality that $t_i\geq(k+1)r$ for all $i$.

Since $|u(t)|\leq\delta$ for all $t\geq 0$,
and $|\phi(t)|<\delta e^{-Lkr}\leq \delta$ for $t\leq0$
it follows from Assumption~\ref{Assume:DDE}
items 2 and 4 that $|\frac{d}{dt}u(t)|\leq (L_0+L_1)\delta = L\delta$ for all $t\geq0$.

Now, consider the sequence of functions $v_i(\theta)=u(t_i+\theta)$ for $\theta\in[-(k+1)r,0]$.
These functions and their derivatives are uniformly bounded with $\|v_i\|\leq\delta$ and
$\|\tfrac{d}{dt}v_i\|\leq L\delta$.
The set of all $C^1$ functions satisfying these bounds forms a
uniformly bounded
and equicontinuous closed family of functions defined on compact
set $[-(k+1)r,0]$. By 
the Arzel\`a-Ascoli theorem
the sequence of functions $v_i(\theta)$
has a uniformly convergent subsequence. Let
$\{v_i\}$ now denote this subsequence
and let $v(\theta)$
be the limiting function, which has Lipschitz constant $L\delta$, and satisfies $v(0)=x_\infty$.
Note that $|v(\theta)|\leq\delta_\infty$ for all $\theta\in[-(k+1)r,0]$, since the existence of a point
with $|v(\theta)|>\delta_\infty$ would contradict that
$\limsup _{t\to\infty}|u(t)|= \delta_\infty$.

Let $\phi^*(\theta)=v(-kr+\theta)$ for $\theta\in[-r,0]$ then we claim that the solution of
\eqref{Eq:ScalarOneDelay}
with initial function $\phi^*$
is $u^*(t)=v(t-kr)$ for $t\in[0,kr]$. To see that this is true,
let
$\sup_{t\in[0,kr]}|u^*(t)-v(t-kr)|=\epsilon\geq0$.
Now let $u_i(t)$ solve
\eqref{Eq:ScalarOneDelay} with
corresponding initial functions
$\phi_i(\theta)=v_i(-kr+\theta)$ for $\theta\in[-r,0]$,
so $u_i(t)=v_i(t-kr)$ for $t\in[0,kr]$.
For all
$i$ sufficiently large we have $\sup_{t\in[0,kr]}|u_i(t)-v(t-kr)|=\sup_{t\in[0,kr]}|v_i(t-kr)-v(t-kr)|\leq\tfrac13\epsilon$
by the uniform convergence of the $v_i$ to $v$.  But also by the uniform convergence for all $i$ sufficiently large
we have
$|\phi_i(\theta)-\phi^*(\theta)|=|v_i(-kr+\theta)-v(-kr+\theta)|\leq\tfrac13\varepsilon
e^{-L_{\delta}kr}$ for all $\theta\in[-r,0]$,
and hence by Lemma~\ref{Lem:FiniteBoundII}
we have $\sup_{t\in[0,kr]}|u_i(t)-u^*(t)|=\sup_{t\in[0,kr]}|v_i(t-kr)-u^*(t)|\leq\tfrac13\epsilon$.
But now
$$\epsilon=\sup_{t\in[0,kr]}|u^*(t)-v(t-kr)|\leq
\sup_{t\in[0,kr]}|v_i(t-kr)-u^*(t)|+\sup_{t\in[0,kr]}|v_i(t-kr)-v(t-kr)|=\tfrac23\epsilon,$$
which can only be true if $\epsilon=0$ so the solution of
\eqref{Eq:ScalarOneDelay} with
$\phi^*(\theta)=v(-kr+\theta)$ for $\theta\in[-r,0]$
is indeed $u^*(t)=v(t-kr)$ for $t\in[0,kr]$.

Now let $\eta(\theta)=v(\theta-\tau(0,v(\theta)))$ for $\theta\in[-\tau(0,x_\infty),0]$
which implies that $\eta(\theta)=u^*(kr+\theta-\tau(0,u^*(kr+\theta)))$.
Moreover $|v(\theta)|\leq\delta_\infty$ for all $\theta\in[-(k+1)r,0]$ implies
that $|\eta(\theta)|\leq\delta_\infty$ for $\theta\in[-\tau(0,x_\infty),0]$ and hence
$\eta\in\cC^{k-1}\bigl([-\tau(0,x_\infty),0],B(0,\delta_\infty)\bigr)$. To show that
$\eta\in E_{(k)}(\delta_\infty,x_\infty)$ it remains only to show that
$x_\infty\cdot f(0,x_\infty,\eta(0))\geq0$. But if this is false then
\begin{align*}
0> x_\infty\cdot f(0,x_\infty,\eta(0))
& =v(0)\cdot f(0,v(0),\eta(0))= u^*(kr)\cdot f(0,u^*(kr),u^*(kr-\tau(0,u^*(kr))))\\
& =u^*(kr)\cdot \dot{u}^*(kr) = \tfrac12\tfrac{d}{dt}|u^*(kr)|.
\end{align*}
But, $|u^*(kr)|=\delta_\infty$ and $\frac{d}{dt}|u^*(kr)|<0$ implies that there exists $\epsilon>0$ such that
$|u^*(t)|>\delta_\infty$ for $t\in(kr-\epsilon,kr)$, or equivalently
$|v(t)|>\delta_\infty$ for $t\in(-\epsilon,0)$. But this contradicts
$|v(\theta)|\leq\delta_\infty$ for all $\theta\in[-(k+1)r,0]$, so we must have
$x_\infty\cdot f(0,x_\infty,\eta(0))\geq0$ and $\eta\in E_{(k)}(\delta_\infty,x_\infty)$.

Now $v(0)=x_\infty$ implies $v(0) \cdot x_\infty = \delta_\infty^2$. But unless $\delta_\infty=0$
this contradicts that \eqref{odecond} holds for all $\delta\in(0,\delta_1)$. The result follows.
\end{proof}

Notice that  Theorem~\ref{Thm:Technique} not only establishes asymptotic stability
of the steady state, but also shows that the basin of attraction of the steady state contains
the ball
\begin{equation} \label{Eq:basinbd}
\bigl\{\phi:\|\phi\|< \delta_1e^{-Lkr(\delta_1)}\bigr\}.
\end{equation}
We will consider the basin of attraction of the steady state of the model problem \eqref{Eq:1Delay} in Section~\ref{Sec:Basins}.

The proof of Theorem~\ref{Thm:Technique} given above would not be valid for nonautonomous DDEs, 
but would only fail in one crucial step;
for a general nonautonomous DDE \eqref{Eq:GeneralMultipleDelay},
the limiting function $v(t)$ would not in general define a solution of the DDE. The result is easily extended
to periodically nonautonomous DDEs by choosing the initial sequence $t_i$ to be
$t_i=(k+1)r+iT$ where $T$ is the period of 
$f$, and if necessary taking a subsequence so that $u(t_i)$ converges to $x_\infty$.

Our asymptotic stability result and its proof differs very significantly from other asymptotic stability results for RFDEs which are all similar to Theorem 4.2 of Hale and Verduyn Lunel \cite{HaleLunel:1}.
Beyond the technical differences in continuity assumptions, and whether delays are locally or globally bounded, there are two fundamental but related differences between our result and results such as those in \cite{HaleLunel:1}. Firstly, in Theorem~\ref{Thm:Technique} we establish \emph{asymptotic stability}, but in
Theorem 4.2 of \cite{HaleLunel:1} the stronger property
of \emph{uniform asymptotic stability} is obtained. But secondly,
auxiliary functions with specific properties are required (in Theorem 4.2 of \cite{HaleLunel:1} four auxiliary functions, $u$, $v$, $\omega$ and $p$ appear)
to obtain the contraction that leads to the uniform asymptotic stability.
Construction of such functions
is difficult even for constant delay DDEs, and a major obstacle to the application of these theorems. In contrast,
we use a proof by contradiction which shows that there does not exist a trajectory which is not asymptotic
to the steady state. The contradiction argument establishes asymptotic stability rather than
uniform asymptotic stability, but does not require any troublesome auxiliary functions, and
thus is much
easier to apply. In the following sections we will use Theorem~\ref{Thm:Technique} to study
the asymptotic stability of the steady state of the model state-dependent DDE \eqref{Eq:1Delay}.


We next define the larger sets containing $E_{(k)}(\delta,x)$ in which
we will later show that conditions of Theorem~\ref{Thm:Technique} hold
to establish asymptotic stability for the model problem \eqref{Eq:1Delay}.
By items 4--6 in Assumption~\ref{Assume:DDE}, if a bound on $u(t)$ is given for $t\in[-r(\delta),(k-1)r(\delta)]$ we can also find bounds on up to the $k-1$ order derivatives of $u\bigl(t-\tau(0,u(t)\bigr)$ for $t\in[(k-1)r(\delta),kr(\delta)]$. These bounds can be derived from the bounds on $f$, $\tau$ and their derivatives. Recalling the definition of $E_{(k)}(\delta,x)$
in Definition~\ref{Defn:E} this leads us to the following definition.

\begin{defn} \label{Defn:ED}
Suppose that Assumption~\ref{Assume:DDE} is satisfied for \eqref{Eq:ScalarOneDelay} and $k\geqslant 1$. Let $\delta\in(0,\delta_0]$ while $|x|=\delta$. Let the functions $\mathcal{D}_j(\delta)
$ be Lipschitz continuous in $\delta$ for
$j=0,\dotsc,k-1$ and satisfy
\begin{equation} \label{Eq:Dj}
\mathcal{D}_{j}(\delta)\geqslant
\sup_{t\in[(k-1)r(\delta),kr(\delta)]} \bigl|\tfrac{d^j}{dt^j}u\big(t-\tau(t,u(t))\big)\bigr|,
\end{equation}
given that $|u(t)| \leqslant \delta$ for all $t\in[-r(\delta),kr(\delta)]$, where $u(t)$ is a solution to \eqref{Eq:ScalarOneDelay}.
Define the set
\begin{equation} \label{Eq:calEkdx}
\mathcal{E}_{(k)}(\delta,x)=\left\{\begin{array}{l}
 \eta:  \eta \in\PC^{k-1}\big([-\tau(0,x),0], B(0,\delta)\big), \;
x \cdot f(0,x,\eta(0)) \geqslant 0,\\
\qquad \bigl|\tfrac{d^j}{d\theta^j}\eta(\theta)\bigr|\leqslant \mathcal{D}_{j}(\delta)
\text{ for } \theta\in[-\tau(0,x),0], j=0,\dotsc,k-1
\end{array}\right\}
\end{equation}
where $\PC^{k-1}\big([-\tau(0,x),0], B(0,\delta)\big)$ denotes the space of $C^{k-2}$ functions which
are piecewise $C^{k-1}$.
\end{defn}

Clearly, $E_{(k)}(\delta,x)\subseteq \mathcal{E}_{(k)}(\delta,x)$. It is convenient to consider piecewise $C^{k-1}$ functions
in Definition~\ref{Defn:ED} because we will later seek the supremum of an integral over the set $\mathcal{E}_{(k)}(\delta,x)$. Even if all the functions in $\mathcal{E}_{(k)}(\delta,x)$ were 
$C^{k-1}$, in general the maximiser could still be piecewise $C^{k-1}$.

In Section~\ref{Sec:Stability:k=123}
we derive bounds $\mathcal{D}_{j}(\delta)$ for the model problem \eqref{Eq:1Delay}, and use these to
identify parameter regions for which all $ \eta \in \mathcal{E}_{(k)}(\delta,x)$ satisfy
\eqref{odecond}, and hence the steady state of \eqref{Eq:1Delay}
is asymptotically stable by Theorem~\ref{Thm:Technique}.
For $\delta\in(0,\delta_0]$, $x\in\mathbb{R}^d$ and $|x|=\delta$,
it is
useful to define
\begin{equation} \label{eq:GF}
\mathcal{G}(\delta,x) = \sup_{\eta\in \mathcal{E}_{(k)}(\delta,x)} \frac{1}{\delta} v(x,\eta)(0)\cdot x, \qquad
\mathcal{F}(\delta) = \sup_{|x|=\delta} \, \mathcal{G}(\delta,x),
\end{equation}
where $v(x,\eta)$ is the solution to \eqref{Eq:AuxiliaryODE}.
Notice that for $\delta\in(0,\delta_0]$ and $|x|=\delta$ we have
\begin{equation} \label{stronger}
\sup_{\eta\in E_{(k)}(\delta,x)} \frac{1}{\delta}v(x,\eta)(0)\cdot x
\leqslant \sup_{\eta\in \mathcal{E}_{(k)}(\delta,x)} \frac{1}{\delta}v(x,\eta)(0)\cdot x = \mathcal{G}(\delta,x) \leqslant \sup_{|x|=\delta} \mathcal{G}(\delta,x)  = \mathcal{F}(\delta).
\end{equation}
Thus if $\mathcal{F}(\delta)<\delta$ for all $\delta\in(0,\delta_1)$ then
\eqref{odecond} holds for all $\delta\in(0,\delta_1)$
and Theorems~\ref{Thm:Lyapunov_NA} and~\ref{Thm:Technique} can be applied. Although $\mathcal{F}(\delta)<\delta$
is a somewhat stronger condition than \eqref{odecond} we will find it
convenient to work with when considering the model problem \eqref{Eq:1Delay}.

The set $E_{(k)}(\delta,x,t)$ given by \eqref{eq:Ekdxt} for the DDE
\eqref{Eq:GeneralMultipleDelay} can be easily generalised to a larger set $\mathcal{E}_{(k)}(\delta,x,t)$,
in a similar manner.
For $t\geq t_0+kr(\delta)$ we let
\begin{equation} \label{Eq:calEkdxt}
\mathcal{E}_{(k)}(\delta,x,t)=\left\{\begin{array}{l}
(\eta_1,\ldots,\eta_N):  \eta_i \in\PC^{k-1}\big([-r(\delta),0], B(0,\delta)\big), \\
x \cdot f(t,x,\eta_1(0),\ldots,\eta_N(0)) \geqslant 0,\\
\bigl|\tfrac{d^j}{d\theta^j}\eta_i(\theta)\bigr|\leqslant \mathcal{D}_{ij}(\delta,t)
\text{ for } \theta\in[-r(\delta),0], i=1,\dotsc,N, \; j=0,\dotsc,k-1
\end{array}\right\}
\end{equation}
where for all solutions $u$ to \eqref{Eq:GeneralMultipleDelay} which satisfy $|u(s)|\leq\delta$
for $s\in[t-(k+1)r(\delta),t]$,
\begin{equation} \label{Eq:Dij}
\mathcal{D}_{ij}(\delta,t)\geqslant
\sup_{s\in[t-r(\delta),t]} \Bigl|\tfrac{d^j}{dt^j}u\big(s-\tau_i(s,u(s))\big)\Bigr|.
\end{equation}
It follows that $E_{(k)}(\delta,x,t)\subseteq\mathcal{E}_{(k)}(\delta,x,t)$,
and hence establishing properties on the set $\mathcal{E}_{(k)}(\delta,x,t)$ is sufficient to apply
Theorem~\ref{Thm:Lyapunov_NA}. However, we will consider the 
autonomous model problem \eqref{Eq:1Delay}  in the following sections, and so will not need to consider
$E_{(k)}(\delta,x,t)$ or $\mathcal{E}_{(k)}(\delta,x,t)$ further.

%% file: ModelEquationProperties.tex
In the following sections we will apply the Lyapunov-Razumikhin theory of Section~\ref{Sec:Extension}
to the model state-dependent DDE given in \eqref{Eq:1Delay}.
In this section we consider the properties of the DDE \eqref{Eq:1Delay} and its auxiliary ODE \eqref{Eq:AuxiliaryODE},
and will define the sets and functions that we will use to apply our
results to this model problem.
We begin by considering boundedness and, existence and uniqueness of solutions of the DDE \eqref{Eq:1Delay}
with $\mu+\sigma<0$,
which generalise the results of Mallet-Paret and Nussbaum in \cite{MalletParetNussbaum:4} for $\sigma<\mu<0$.

\begin{lem} \label{lem:delayed}
Let $c\ne0$ and $\mu+\sigma<0<a$. If $u\in C^1([0,\infty),\R)$ solves \eqref{Eq:1Delay} for $t\geq0$ with $c\phi(0)\geq-a$ then $t-a-cu(t)<t$ for all $t>0$.
\end{lem}

\begin{proof}
The model DDE \eqref{Eq:1Delay} is invariant under the transformation $u\mapsto -u$, $c\mapsto -c$, so we consider only the case $c>0$. 
Suppose first that $\phi(0)>-a/c$ and assume there exists $t^*>0$ for which $u(t^*)=-a/c$ for the first time. 
Since $u(t)>-a/c$ for $0\leq t<t^*$ implies $\dot{u}(t^*)\leq0$, but from \eqref{Eq:1Delay} with $u(t^*)=-a/c$ we have $\dot{u}(t^*)=(\mu+\sigma)u(t^*)=-\tfrac{a}{c}(\mu+\sigma)>0$, supplying a contradiction. Thus we must have $u(t)>-\frac{a}{c}$ for all $t\geq 0$ and the result follows.
If $\phi(0)=-a/c$ then $\dot{u}(0)>0$ and the result follows similarly.
\end{proof}

We will always consider the DDE \eqref{Eq:1Delay} with $a>0$ and $\mu+\sigma<0$, then
Lemma~\ref{lem:delayed} assures that the deviating argument is always a delay.
The lemma also gives the lower bound $u(t)>-a/c$ on solutions when $c>0$ (or an upper bound on solutions when $c<0$).
When $\mu<0$ we can bound solutions above and below. It is convenient to define
\begin{equation} \label{Eq:1Delay:L0M0tau0}
M_0=-\frac{a}{c}, \qquad N_0=\frac{a\sigma}{c\mu}, \qquad \tau=a+cN, \qquad \tau_0=a+cN_0.
\end{equation}
Here $N$ is used to denote a bound on solutions of the single delay DDE \eqref{Eq:1Delay} (an upper bound if $c>0$ and a lower bound if $c<0$), in contrast to the multiple delay DDE \eqref{Eq:GeneralMultipleDelay} for which we used $N$ to denote the number of delays.
We use $M$ to denote the other bound on the solution, and the notation $\text{sign}(c)[M,N]$ to denote $[M,N]$ if $c>0$ and $[-N,-M]$ if $c<0$ (and similar notation for open and half-open intervals).

\begin{lem} \label{lem:1delbound}
Let $c\ne 0$ and $\mu+\sigma<0<a$ and $\mu<0$. Suppose $u\in C^1([0,\infty),\R)$ solves \eqref{Eq:1Delay} for $t\geq0$.
If $\sigma>0$ let $\textit{sign}(c)M\in\textit{sign}(c)[M_0,0)$ and $\textit{sign}(c)N>0$,
and suppose that
$\textit{sign}(c)\phi(t)\in\textit{sign}(c)[M,N]$ for all $t\in[-\tau,0]$. 
If $\sigma\leq0$ let $M=M_0$ and $N=\max\{N_0,\phi(0)\}$ and suppose
$\textit{sign}(c)\phi(t)\geq\textit{sign}(c)M_0$ for all $t\in[-\tau,0]$.
Then
\begin{equation} \label{eq:ubd}
\textit{sign}(c)u(t)\in\textit{sign}(c)(M,N), \quad \forall t>0.
\end{equation}
\end{lem}

\begin{proof}
Again, we consider the $c>0$ case, then it is sufficient to show that $\dot{u}(t)>0$ if $u(t)=M$, and
$\dot{u}(t)<0$ if $u(t)=N$ given that $u(s)\in(M,N)$ for $s\in(0,t)$. 
The case where $u(t)=M_0$ 
is dealt with in the proof of Lemma~\ref{lem:delayed}, the other
cases are straightforward.
\end{proof}

\begin{thm} \label{Thm:1Delay_existence_uniqueness}
Let $\mu+\sigma<0<a$.
Let the initial history function $\phi$ be 
continuous and for $\mu<0$ satisfy
the bounds given
in Lemma~\ref{lem:1delbound}.
For $\mu\geq0$
let $\textit{sign}(c)\phi(t)\geq\textit{sign}(c)M_0$ for all $t\in(-\infty,0]$.
Then there exists at
least one solution $u\in C^1([0,\infty),\R)$ which solves \eqref{Eq:1Delay} for all $t\geq0$.
If $\mu<0$ any solution satisfies the bounds \eqref{eq:ubd}, while if $\mu\geq0$ any solution
satisfies $\textit{sign}(c)u(t)\geq\textit{sign}(c)M_0$ for all $t\geq0$.
If $\phi$ is locally Lipschitz the solution is unique.
\end{thm}

\begin{proof}
Local existence and uniqueness follows directly from the results of Driver~\cite{Driver:1}, and for $\mu<0$
global existence and uniqueness follows from the extended existence result of Driver~\cite{Driver:1} using
the bounds on the delay and solution given by Lemma~\ref{lem:delayed} and~\ref{lem:1delbound}. The only delicate
case is for $-\sigma>\mu>0$ for which Lemma~\ref{lem:delayed} gives one of the bounds, $\text{sign}(c)u(t)\geq M_0$.
We consider this case with $c>0$. Then,
$\dot{u}(t)\leq \mu u(t)+\sigma M_0$ and the Gronwall lemma implies that
\begin{equation} \label{eq:4thquadexpbd}
u(t)\leq\bigl(\phi(0)+\tfrac{\sigma}{\mu}M_0\bigr)e^{\mu t}-\tfrac{\sigma}{\mu}M_0
=(\phi(0)-N_0)e^{\mu t}+N_0.
\end{equation}
Since $\phi(0)\geq M_0>N_0$ in this case,
solutions cannot become unbounded in finite time, and global existence again follows. For this case
$\phi(t)$ should be defined for all $t\leq0$ since with the
exponentially growing bound \eqref{eq:4thquadexpbd} on $u(t)$ it is possible that $t-a-cu(t)\to-\infty$ as $t\to+\infty$.
\end{proof}

The constant delay DDE  which corresponds to \eqref{Eq:1Delay} with $c=0$, known as Hayes equation, has
been much studied.  The
$(\mu,\sigma)$ values for which its steady state
is asymptotically stable when $a>0$ and $c=0$ are well known (see eg. \cite{HaleLunel:1}) and given in Definition~\ref{Defn:StabilityRegion}.

\begin{defn}[Stability region $\Sigma_\star$]
\label{Defn:StabilityRegion}
Let $a>0$ and $c=0$. Let $\Sigma_\star$ be the open set of the $(\mu,\sigma)$-parameter space between the curves
\[ \ell_\star = \bigl\{(s,-s) : s\in(-\infty,1/a]\bigr\},
\quad g_\star= \bigl\{(\mu(s),\sigma(s)) : s\in(0,\pi/a)\bigr\} \]
where the functions $\mu(s)$ and $\sigma(s)$ are given by
\begin{equation}  \label{Eq:musigma-bounds}
  \mu(s)  =  s\cot(as), \quad  \sigma(s)  = - s\csc(as).
\end{equation}
The region $\Sigma_\star$ is further divided into three subregions: the cone $\cone=\left\{ (\mu,\sigma) : |\sigma|<-\mu\right\}$, the wedge $\wedg=(\Sigma_\star \setminus \cone )\cap \left\{\mu< 0\right\}$ and the cusp $\cusp=\Sigma_\star\cap \left\{\mu\geq0\right\}$, which are shown in Figure~\ref{Fig:ConstantDelay_divisions}.
\end{defn}

\begin{figure}[t]
\begin{center}
\includegraphics[scale=0.4]{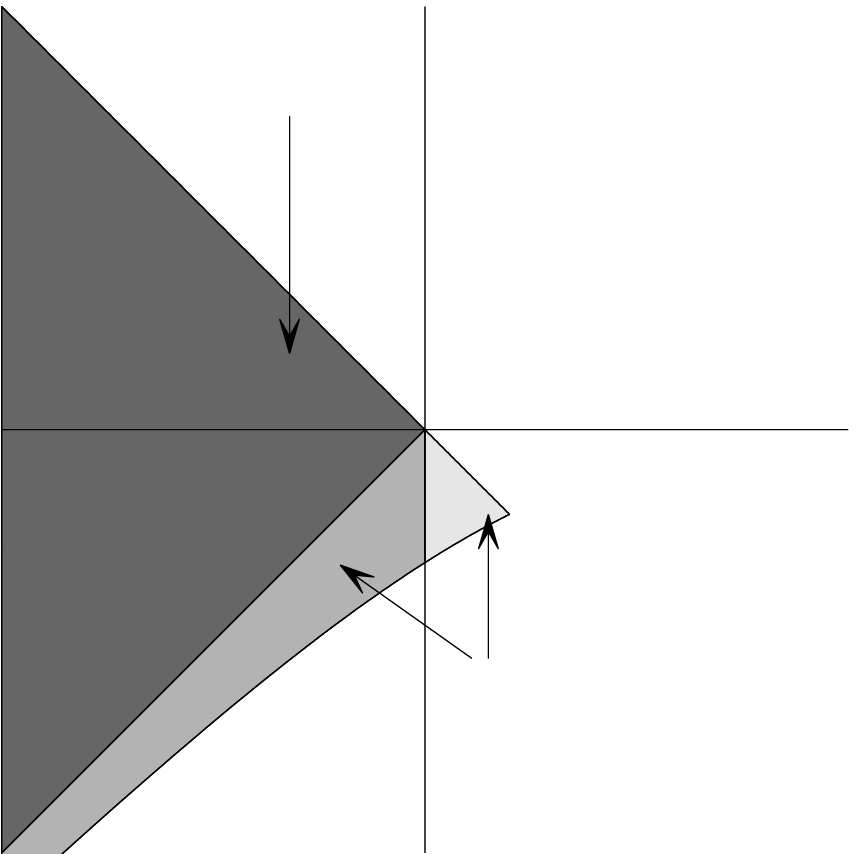}
\put(-75,30){{\footnotesize delay-dependent}}
\put(-145,145){{\footnotesize delay-independent}}
\put(-7,77){{\footnotesize $\mu$}}
\put(-90,1){{\footnotesize $\sigma$}}
\put(-150,60){{\footnotesize $\cone$}}
\put(-150,105){{\footnotesize $\cone$}}
\put(-80,67){{\footnotesize $\cusp$}}
\put(-120,38){{\footnotesize $\wedg$}}
\FigureCaption{\label{Fig:ConstantDelay_divisions}The analytic stability region $\Sigma_\star$ in the $(\mu,\sigma)$ plane, divided into the delay-independent cone $\cone$, and the delay-dependent
  wedge $\wedg$ and cusp $\cusp$.}
\end{center}
\end{figure}

The region $\Sigma_\star$ is the parameter region in the $(\mu,\sigma)$-plane for which the zero solution to the DDE \eqref{Eq:1Delay} is locally asymptotically stable
in both the constant and state-dependent delay cases.
The cone $\cone$ forms the delay-independent stability region (because this does not change when $a$ is changed) while $\wedg\cup\cusp$ is often referred to as the delay-dependent stability region.
For the constant delay case ($c=0$)
this region is found from the characteristic equation \cite{ElsgoltsNorkin:1}. The results of
Gy\"ori and Hartung \cite{GyoriHartung:1} show the state-dependent case ($c\ne 0$) of \eqref{Eq:1Delay} has the same (exponentially) asymptotic stability region. On the boundary of $\Sigma_\star$ the steady-state is
Lyapunov stable for the constant delay case, and the stability is delicate in the state-dependent case
\cite{Stumpf:1}.

In this paper we derive new proofs of stability in parts of $\Sigma_\star$ for the state-dependent case using Theorem~\ref{Thm:Technique}.
The asymptotic stability of the zero solution to \eqref{Eq:1Delay} in all of the delay-independent region
($(\mu,\sigma)\in\cone$) will be shown in Theorem~\ref{Thm:localcone}.
In
Theorem~\ref{Thm:StabilityDDE:Razumikhin_k=123}
we will also show asymptotic stability of the steady
state of the model problem \eqref{Eq:1Delay} for $(\mu,\sigma)$ in subsets of $\wedg\cup\cusp$, by applying
Theorem~\ref{Thm:Technique} with $k=1$ to $3$.
Here we define some notation that will be required. Let $(\mu,\sigma)\in\wedg\cup\cusp$, $k\in\mathbb{Z}$,
$k\geqslant 1$, $\delta_0\in (0,|a/c|)$ and $\delta\in(0,\delta_0)$. It is easy to see that
Assumption~\ref{Assume:DDE} is satisfied for \eqref{Eq:1Delay} with $L_0=|\mu|$, $L_1=|\sigma|$,
$\tau_{\max}=a$ and $r(\delta)=a+|c|\delta$.
Thus for the model problem  \eqref{Eq:1Delay} the sets $E_{(k)}(\delta,x)$ from Definition~\ref{Defn:E} are given by
\begin{equation}  \label{eq:Ekdxmodeq}
E_{(k)}(\delta,x)=\left\{\begin{array}{l}
\eta:  \eta \in \cC^{k-1}([-a-cx,0],[-\delta,\delta]), \; \mu x^2 + \sigma x \eta(0) \geqslant 0,  \\
 \mbox{}\qquad \text{and for some initial function $\varphi\in\cC$ equation \eqref{Eq:ScalarOneDelay} has solution}\\
\mbox{}\qquad\quad \eta(\theta) = u(kr(\delta)+\theta-a-cu(kr(\delta)+\theta)) \text{ for } \theta\in[-a-cx,0]
\end{array}\right\}.
\end{equation}
To apply the stability theorems in the next section
we will derive some bounds $\mathcal{D}_j(\delta)$ for $j=0,\dotsc,k-1$ and $\delta\in(0,\delta_0]$ as in Definition~\ref{Defn:ED}.
Once these bounds are determined, the sets $\mathcal{E}_{(k)}(\delta,x)$ from Definition~\ref{Defn:ED} are given by
\begin{equation}
\mathcal{E}_{(k)}(\delta,x)=\left\{\begin{array}{l}
\eta:  \eta \in \PC^{k-1}([-a-cx,0],[-\delta,\delta]), \; \mu x^2 + \sigma x \eta(0) \geqslant 0,  \\
\quad \bigl|\tfrac{d^j}{d\theta^j}\eta(\theta) \bigr|\leqslant \mathcal{D}_j(\delta)
\text{ for } \theta\in[-a-cx,0], j=0,\dotsc,k-1
\end{array}\right\}.
\label{Eq:E1Delay}
\end{equation}
Since the DDE \eqref{Eq:1Delay} is scalar the set of $x$ such that $|x|=\delta$ consists of just
two points $x=\delta$ and $x=-\delta$.
Suppose first that $x=\delta$, then \eqref{Eq:E1Delay} implies that
$\eta(0)\in[-\delta,-\delta\mu/\sigma]$.
Let $r_+=a+c\delta$ so that we can write the auxiliary ODE problem \eqref{Eq:AuxiliaryODE}
becomes
\begin{equation}\left\{
\begin{array}{ll}
   \dot v(\theta) =\mu v(\theta)+ \sigma \eta(\theta), & \theta\in[-r_+,0],\\
   v(-r_+)=\eta(0).&
 \end{array}\right.
 \label{Eq:AuxiliaryODESDD}
\end{equation}
Integrating \eqref{Eq:AuxiliaryODESDD} yields,
\begin{equation} \label{Eq:FullIntegration1}
v(0)=  \eta(0)e^{\mu r_+}  + \sigma \! \int_{-r_+}^{0} {e^{-\mu \theta}\eta(\theta)d\theta}.
\end{equation}

\begin{defn}
\label{Defn:StabilityDDE:etaIP}
Let $a>0$, $c\ne0$, $\sigma\leqslant\mu$ and $\sigma<-\mu$. For any $\delta\in(0,|a/c|)$
and $\hat u \in[-\delta,-\delta\mu/\sigma]$,
define $r_+=a+c\delta$ and $\eta_{(k)}(\theta)$ for $\theta\in[-r_+,0]$ by
\begin{equation}  \label{Eq:etak}
\eta_{(k)}(\theta) =  \inf_{\begin{subarray}{c}\eta\in \mathcal{E}_{(k)}(\delta,\delta)\\
\eta(0)=\hat u \end{subarray}}\eta(\theta).
\end{equation}
We also define the function $\mathcal I(\hat u,\delta,c,k)$ to be
\begin{equation} \label{Eq:Ick}
 \mathcal I(\hat u,\delta,c,k)= \hat{u} e^{\mu r_+}  + \sigma\!
 \int_{-r_+}^{0} {e^{-\mu\theta}\eta_{(k)}(\theta)d\theta}.
\end{equation}
\end{defn}

The function $\eta_{(k)}$ given by \eqref{Eq:etak}
is the most negative one
in $\mathcal{E}_{(k)}(\delta,\delta)$ satisfying $\eta(0)=\hat u$, and so
since $\sigma<0$, this function
maximizes $v(0)$ for fixed $\eta(0)$ by maximising the second term in \eqref{Eq:FullIntegration1}.
This is the reason for considering
$\eta \in \PC^{k-1}([-a-cx,0],[-\delta,\delta])$ in the definition of $\mathcal{E}_{(k)}(\delta,x)$.
We really want to maximise $v(0)$ for $\eta \in E_{(k)}(\delta,x)$, where the smaller set
$E_{(k)}(\delta,x)$ is defined in \eqref{eq:Ekdx}. Even though all the functions $\eta\in E_{(k)}(\delta,x)$
satisfy $\eta \in \cC^{k-1}([-a-cx,0],[-\delta,\delta])$, the maximiser will in general only be piecewise $\cC^{k-1}$.
With Definition~\ref{Defn:StabilityDDE:etaIP} we can derive bounds on the solution $v(0)$ of the auxiliary ODE \eqref{Eq:AuxiliaryODESDD} for
all $\eta\in\mathcal{E}_{(k)}(\delta,x)$ in both cases where $x=\pm\delta$.

\begin{lem}
\label{Lem:v<P}
Let $a>0$, $c\ne0$, $\sigma\leqslant\mu$ and $\sigma<-\mu$. Let $\delta\in(0,|a/c|)$.
The solution of the auxiliary ODE system \eqref{Eq:AuxiliaryODESDD}
satisfies
$$v(0)\leq \sup_{\hat u \in \left[-\delta,-\frac{\mu}{\sigma}\delta\right]}\!\!\mathcal{I}(\hat u,\delta,c,k), \;
\forall\eta\in\mathcal{E}_{(k)}(\delta,\delta), \quad
v(0)\geq -\!\!\sup_{\hat u \in \left[-\delta,-\frac{\mu}{\sigma}\delta\right]}\!\!\mathcal{I}(\hat u,\delta,-c,k), \;
\forall\eta\in\mathcal{E}_{(k)}(\delta,-\delta).$$
\end{lem}

\begin{proof}
First consider $x=\delta$. The function $\mathcal I(\hat u,\delta,c,k)$ comes from \eqref{Eq:FullIntegration1}
and depends on $c$ and $\delta$ through $r_+$.
Since,
as noted above, the choice
of $\eta_{(k)}$
maximizes
\eqref{Eq:Ick} for fixed $\hat u$,
the first inequality in the statement of the lemma follows.

Next consider $x=-\delta$, then \eqref{Eq:E1Delay} implies that
$\eta(0)\in[\delta\mu/\sigma,\delta]$.
This time we should consider the most positive function in
$\mathcal{E}_{(k)}(\delta,-\delta)$ satisfying $\eta(0)=\hat u\in[\delta\mu/\sigma,\delta]$,
to obtain a lower bound on $v(0)$ for all $\eta\in\mathcal{E}_{(k)}(\delta,-\delta)$. However,
the model DDE \eqref{Eq:1Delay} is invariant under the transformation $(u,c)\mapsto(-u,-c)$,
so this function is $-\eta_{(k)}(\theta)$ and the second inequality follows.
\end{proof}

Notice from \eqref{Eq:Ick} that the functions $\mathcal{I}(\hat u,\delta,c,k)$ and $\mathcal{I}(\hat u,\delta,-c,k)$
only differ in their integration limits with $\mathcal{I}(\hat u,\delta,c,k)$ integrating $\eta_{(k)}$ over the
interval $[-a-c\delta,0]$ and $\mathcal{I}(\hat u,\delta,-c,k)$ integrating over $[-a+c\delta,0]$. The integration
over the larger of these intervals will be important in the following sections and so it is convenient to define
\begin{equation} \label{Eq:P}
P(\delta,c,k)=
\sup_{\hat u \in[-\delta,-\delta\mu/\sigma]}\mathcal{I}(\hat u,\delta,|c|,k).
\end{equation}
Comparing the cases when $x=\delta$ and $-\delta$ has to be done separately for each value of $k$, and we can also
explicitly define the functions $\eta_{(k)}$ for each $k$. This is handled in the following section where we show that $P(\delta,c,k)<\delta$ implies $\mathcal{F}(\delta)<\delta$, and
apply Theorem~\ref{Thm:Technique} to obtain asymptotic stability for
$\bigl\{(\mu,\sigma) : P(\delta,c,k)<\delta\bigr\}$.

Barnea \cite{Barnea:1} applied Lyapunov-Razumikhin techniques
to the $c=0$ constant delay case of the model DDE \eqref{Eq:1Delay}.
His results do not apply to state-dependent case, as they were based on a result for
autonomous RFDEs which assumed $F$ was Lipschitz, and he did not define an
auxiliary ODE, nor sets similar to $E_{(k)}(\delta,x)$ or $\mathcal{E}_{(k)}(\delta,x)$.
However, he did define functions $\eta_{(k)}$ for the constant delay case by considering the most
negative bounded functions with $k-1$ bounded derivatives as the function segments in the RFDE.
In the limit as $c\to0$ our $\eta_{(k)}$ functions reduce to those found by Barnea for the constant delay case.
Because of the linearity of \eqref{Eq:1Delay} with $c=0$, Barnea did not have to consider the upper and
lower bounds separately as we did in Lemma~\ref{Lem:v<P}, but did define a function which is equivalent to
$P(\delta,0,k)$ in \eqref{Eq:P}. Our asymptotic stability
results for the state-dependent model DDE \eqref{Eq:1Delay} constitute a significant
generalisation of the Lyapunov stability results of Barnea \cite{Barnea:1} for the constant delay case, and moreover
in Section~\ref{Sec:StabilityDDE:Measurements}
we will correct an error of Barnea for the $k=2$ constant delay case.

%% file: ModelEquation123.tex
In this section we consider the model state-dependent DDE \eqref{Eq:1Delay} and
use Theorem~\ref{Thm:Technique} to show that the steady state is asymptotically stable
in various parameter sets.  In Theorem~\ref{Thm:localcone} we use the set $\mathcal{E}_{(1)}(\delta,x)$
to show that the steady state is asymptotically stable
whenever the parameters values $(\mu,\sigma)$ are in the cone $\cone$. In the rest of the section
we consider parameters in the wedge and the cusp ($\wedg\cup\cusp$),
and use Theorem~\ref{Thm:Technique} with $k=1$, $2$ and $3$ to
show the steady state is asymptotically stable
for the parameter sets $\bigl\{(\mu,\sigma):P(1,0,k)<1\bigr\}$,
where $P(\delta,c,k)$ is defined in \eqref{Eq:P}.
We will often write $\bigl\{P(1,0,k)<1\bigr\}$ as shorthand for these sets which
are nested in $\wedg\cup\cusp$, becoming larger with $k$.
We also find lower bounds on the basin of attraction of the steady state.
For the constant delay case ($c=0$) the parameter regions
found in $\wedg\cup\cusp$ are independent of the choice of $\delta$ in $\mathcal{E}_{(k)}(\delta,x)$.
For the state-dependent case, these regions change with $c$ and $\delta$ (see Figure~\ref{Fig:StabilityExample_k=12})
and converge to the region for the constant delay case as $\delta\to 0$.

We begin by showing asymptotic stability in the cone $\cone$.
The following result could also be shown by adapting a stability result for time-dependent delays, such as that of Yorke \cite{Yorke:1}.
Recall that $M_0$ is defined by \eqref{Eq:1Delay:L0M0tau0}.

\begin{thm}[Asymptotic stability for \eqref{Eq:1Delay} in $\cone$] \label{Thm:localcone}
Let $a>0$, $c\ne0$ and $|\sigma|<-\mu$ so $(\mu,\sigma)\in\cone$.
If $|\varphi(t)|\leq|M_0|$ for $t\in[-a-c|M_0|,0]$ then the solution $u(t)$ to \eqref{Eq:1Delay} satisfies $u(t)\to 0$ as $t\to\infty$.
\end{thm}

\begin{proof} 
With $|x|=\delta$ and $\mu<-|\sigma|$ it is impossible to satisfy $\mu x^2 + \sigma x \eta(0) \geqslant 0$
with $|\eta(0)|\leq\delta$. Thus $\mathcal{E}_{(1)}(\delta,x)$ and also $E_{(1)}(\delta,x)$ are empty, and asymptotic stability of the steady state
follows from Theorem~\ref{Thm:Technique}. This holds for all $\delta\in(0,|M_0|]$ and it follows directly from
Theorem~\ref{Thm:Technique} that $u(t)\to 0$ as $t\to\infty$ provided
$|\varphi(t)|<|M_0|e^{-Lr(|M_0|)}$ for $t\in[-a-c|M_0|,0]$. However, the exponential correction term $e^{-Lr(|M_0|)}$ comes
from using Lemma~\ref{Lem:FiniteBoundI} in the proof of Theorem~\ref{Thm:Technique} to ensure that $|u(t)|<|M_0|$
for $t\in[0,r(|M_0|)]$. But Lemma~\ref{lem:1delbound} already ensures that $|u(t)|<|M_0|$ for
all $t>0$ if $|\varphi(t)|\leq|M_0|$ for the model DDE \eqref{Eq:1Delay} (note that in $\cone$ we have $|M_0|>|N_0|$); the result follows.
\end{proof}


We already noted in Section~\ref{Sec:ModelEquation} that Assumption~\ref{Assume:DDE} is satisfied
for \eqref{Eq:1Delay}, and
derived an expression for $\mathcal{E}_{(k)}(\delta,x)$
and indicated which are the most relevant functions in these sets.
For $k=1$ we do not need any bounds $\mathcal{D}_j(\delta)$ and the members of the set
$\mathcal{E}_{(1)}(\delta,x)$ need not be continuous. Then the function $\eta_{(1)}$
from \eqref{Eq:etak} is given by,
\begin{equation}
 \eta_{(1)}(\theta) =  \left\{ {\begin{array}{rl}
    -\delta , & \theta\in[-a-c\delta,0),  \\
   \hat{u}, & \theta=0.
 \end{array} }\right.\label{Eq:eta1}
\end{equation}

For $k=2$ we need to
find $\mathcal{D}_1(\delta)$ such that
$|\tfrac{d}{d\theta}\eta(\theta)|\leqslant\mathcal{D}_1(\delta)$ for all
$\eta\in E_{(2)}(\delta,x)$, where $E_{(2)}(\delta,x)$ is defined by \eqref{eq:Ekdxmodeq}.
For $\eta\in E_{(2)}(\delta,x)$ we have
$$\eta(\theta) = u(2r+\theta-a-cu(2r+\theta)) \text{ for } \theta\in[-a-cx,0]\subseteq[-r,0],$$
where $|u(t)|\leq\delta$ for $t\in[-r,2r]$ and solves \eqref{Eq:1Delay} for $r\geq0$.
We easily derive that
$|\dot{u}(t)|\leq |\mu u(t)|+|\sigma u(t-a-cu(t))| \leq (|\mu| +|\sigma|)\delta$ for $t\in[0,2r]$.
Then
$$\eta'^{\!}(\theta)=\tfrac{d}{d\theta}u\bigl(2r+\theta-a-cu(2r+\theta)\bigr)
=\bigl(1-c\dot{u}(2r+\theta)\bigr)\dot{u}\bigl(2r+\theta-a-cu(2r+\theta)\bigr).$$
Hence $|\eta'^{\!}(\theta)|\leq D_1\delta$ for $\theta\in[-r,0]$ where
\begin{equation}  \label{Eq:D1}
D_1 = \bigl(|\mu|+|\sigma|\bigr)\bigl(1+(|\mu|+|\sigma|)|c|\delta\bigr).
\end{equation}
Thus we can choose $\mathcal{D}_1(\delta)=D_1\delta$ (note that $D_1$ also depends on $\delta$) to
define $\mathcal{E}_{(2)}(\delta,x)$ and we obtain that
$E_{(2)}(\delta,x)\subseteq\mathcal{E}_{(2)}(\delta,x)$.
The function $\eta_{(2)}$ is given by
\eqref{Eq:etak}, as 
\begin{equation}  \label{Eq:eta2}
\eta_{(2)}(\theta) = \begin{cases}
\hat u +D_1\delta\theta, & \theta\in\bigl[-\frac{\delta+\hat u}{D_1\delta},0 \bigr],\\
- \delta, & \theta\in\bigl[-a-c\delta,-\frac{\delta+\hat u}{D_1\delta} \bigr],\quad\text{when }
\frac{\delta+\hat u}{D_1\delta}<r_+=a+c\delta.
\end{cases}
\end{equation}
where $D_1$ is defined by \eqref{Eq:D1}.

For $\eta\in E_{(3)}(\delta,x)\subseteq E_{(2)}(\delta,x)$, the same bound on the first derivative of $\eta$ applies, and
we also bound the second derivative as follows. We have
$$\eta(\theta) = u(3r+\theta-a-cu(3r+\theta)) \text{ for } \theta\in[-a-cx,0]\subseteq[-r,0],$$
where $|u(t)|\leq\delta$ for $t\in[-r,3r]$ and solves \eqref{Eq:1Delay} for $r\geq0$. As
above 
we have that  $|\dot{u}(t)|\leq (|\mu|+|\sigma|)\delta$ for $t\in[0,3r]$. Now noting that
$t-a-cu(t)\in[0,2r]$ for $t\in[r,3r]$ we have
$$|\ddot{u}(t)| = \bigl|\mu\dot{u}(t)+\sigma\dot{u}(t-a-cu(t)(1-c\dot u(t))\bigr|
\leq 
(|\mu|+|\sigma|)^2(1+|\sigma c|\delta)\delta,$$
for $t\in[r,3r]$. Then, since $3r+\theta-a-cu(3r+\theta)\in[r,2r]$ for $\theta\in[-r,0]$ it follows that
\begin{align*}
|\eta''^{\!}(\theta)|&=\Bigl|\tfrac{d^2}{d\theta^2}u\bigl(3r+\theta-a-cu(3r+\theta)\bigr)\Bigr|\\
& =\Bigl|\bigl(1-c\dot{u}(3r+\theta)\bigr)^2\ddot{u}\bigl(3r+\theta-a-cu(3r+\theta)\bigr)
-c\ddot{u}(3r+\theta)\bigr)\dot{u}\bigl(3r+\theta-a-cu(3r+\theta)\bigr)\Bigr|\\
& \leq
\bigl(D_1^2 +(|\mu|+|\sigma|)^3|c|\delta\bigr)\bigl(1+|\sigma c|\delta\bigr)\delta
=D_2 \delta.
\end{align*}
Hence for all $\eta\in E_{(3)}(\delta,x)$ we have $|\eta''^{\!}(\theta)|\leq\mathcal{D}_2(\delta)=D_2\delta$
where
\begin{equation}   \label{Eq:D2}
D_2=\bigl(D_1^2 +(|\mu|+|\sigma|)^3|c|\delta\bigr)\bigl(1+|\sigma c|\delta\bigr),
\end{equation}
and $\lim_{\delta\to0}D_2=(\lim_{\delta\to0}D_1)^2=(|\mu|+|\sigma|)^2$.
Taking $D_1$ and $D_2$ to satisfy
\eqref{Eq:D1} and \eqref{Eq:D2} ensures that
$\mathcal{E}_{(3)}(\delta,x)\subseteq E_{(3)}(\delta,x)$.
Then the $\eta_{(k)}$ function from \eqref{Eq:etak} for $k=3$ can be defined by
\begin{equation} \label{Eq:eta3}
\eta_{(3)}(\theta)=\bar \eta_{(3)}(\theta+\theta_{\text{shift}}), \quad \theta\in[-r_+,0]
\end{equation}
where
\begin{equation} \label{Eq:etabarthetashift_k=3}  
\mbox{}\hspace{-0.6em}\bar\eta_{(3)}(\theta) = \left\{ \!\! \begin{array}{ll}
-\delta,                             & \theta\leq0,  \\
-\delta+\frac{\delta}{2}D_2\theta^2, & \theta\in(0,\frac{D_1}{D_2}),\\
-\delta-\frac{\delta D_1^2}{2D_2}+\delta D_{1}\theta, & \theta>\frac{D_1}{D_2}.
 \end{array}  \right.
\hspace{0.5em}
\theta_{\text{shift}} = \left\{ \! \begin{array}{ll}
\left(\displaystyle\frac{2(\hat u + \delta )}{D_2\delta}\right)^{\frac12}\!,
& \hat u\in[-\delta,-\delta+\frac{\delta D_1^2}{2D_2}],\!\! \\
\displaystyle  \frac{\hat u + \delta+ \frac{\delta D_1 ^2}{2D_2}}{D_1\delta},
 & \hat u>-\delta+\frac{\delta D_1^2}{2D_2}.
 \end{array}  \right.  
\end{equation}
Here $\theta_{\text{shift}}$
is a convenient device which allows us to define $\eta_{(3)}(\theta)$ for all values of $\hat u$ by the single
function $\bar\eta_{(3)}(\theta)$ with the shift used to obtain the correct value of $\hat u$.

\begin{figure}[t]
\begin{center}
\subfigure[$\mu=1$, $\sigma=-1.2$, $\hat u =0$]{
\includegraphics[scale=1]{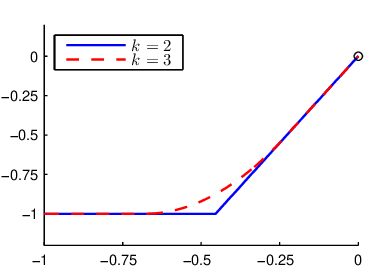}\label{fig:subfig1a}}
\subfigure[$\mu=0.4$, $\sigma=-0.6$, $\hat u =0$]{
\includegraphics[scale=1]{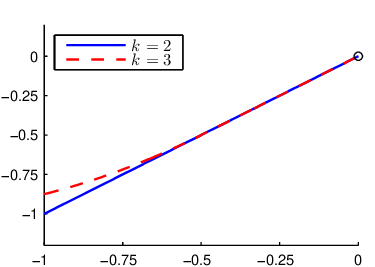}\label{fig:subfig2a}}
\put(-175,110){{\footnotesize $\eta(\theta)$}}
\put(-360,110){{\footnotesize $\eta(\theta)$}}
\put(-28,0){{\footnotesize $\theta$}}
\put(-210,0){{\footnotesize $\theta$}}
\put(-202,102){{\footnotesize $\hat{u}$}}
\put(-17,102){{\footnotesize $\hat{u}$}}
\FigureCaption{Sample $\eta_{(k)}(\theta)$
functions for $a=1$, $c=0$ and $\delta=1$. The $\eta_{(k)}(\theta)$
functions are defined in \eqref{Eq:eta2} and \eqref{Eq:eta3}.}
\label{Fig:etafunctions}
\end{center}
\end{figure}


The $\eta_{(k)}$ functions define $\mathcal I(\hat u,\delta,c,k)$ via equation \eqref{Eq:Ick} and
$P(\delta,c,k)$ through equation \eqref{Eq:P}. For $k=1$, using \eqref{Eq:eta1}
we easily evaluate
\begin{equation} \label{eq:Pdeltac1}
P(\delta,c,1)= \mathcal I\Bigl(-\mu\delta/\sigma,\delta,|c|,1\Bigr) =
\begin{cases}
-\dfrac{\mu}{\sigma}\delta e^{\mu r}  + \delta \dfrac{\sigma}{\mu}(1-e^{\mu r}), & \mu\ne0,\\
-\delta\sigma r, & \mu=0.
\end{cases}
\end{equation}


For $k=2$, from \eqref{Eq:Ick} and \eqref{Eq:eta2}, if $\tfrac{\delta+\hat u}{D_1\delta}\geq r_+$ then
\begin{align}
\mathcal{I}(\hat u,\delta,c,2)   \label{Eq:dI/dtau1def}
& = \hat ue^{\mu r_+}  + \sigma\int_{- r_+}^0 e^{ - \mu\theta}(\hat u + D_1\delta\theta)d\theta  \\
&  = \begin{cases}
\hat u\Bigl[ e^{\mu r_+}  + \frac{\sigma}{\mu}(e^{\mu r_+}- 1) \Bigr]
+ \frac{\sigma}{\mu}D_1\delta\Bigl[\frac{1}{\mu}(e^{\mu r_+}- 1) - r_+ e^{\mu r_+} \Bigr], & \mu\ne0,\\
\hat u  +\sigma r_+\hat u -\frac{\sigma D_1  r_+^2}{2}\delta, & \mu=0,
\end{cases}
\label{Eq:dI/dtau1}
\end{align}
while if $\tfrac{\delta+\hat u}{D_1\delta}< r_+$ then we have to split the integral into two parts and
\begin{align}
\mathcal{I}(\hat u,\delta,c,2)
& = \hat ue^{\mu r_+}  + \sigma \int_{ - r_+}^{-\frac{\delta+\hat u}{D_1\delta}}e^{-\mu\theta}(-\delta)d\theta
+ \sigma \int_{-\frac{\delta+\hat u}{D_1\delta }}^{0} e^{-\mu\theta}(\hat u + D_1\delta\theta)d\theta \label{Eq:dI/dtau2def} \\
&  = \begin{cases}
\hat u \left(e^{\mu r_+}-\frac{\sigma}{\mu}\right)
+ \frac{\sigma}{\mu}\delta\Bigl[\tfrac{D_1}{\mu}\Bigl(e^{\;\mu\tfrac{\delta+\hat u}{D_1\delta}}- 1\Bigr) - e^{\mu r_+} \Bigr], & \mu\ne0, \\
\hat u  -\sigma\delta r_+ +\frac{\sigma}{2D_1\delta}(\delta+\hat u)^2, & \mu=0.
\end{cases}  \label{Eq:dI/dtau2}
\end{align}

To determine $P(\delta,c,2)$ we perform the integration in $\mathcal{I}(\hat u,\delta,|c|,2)$ and find $\hat u$ to
maximise this function.
If $\tfrac{\delta+\hat u}{D_1\delta}\geq r$ then
$\hat u\in\bigl[(r D_1-1)\delta,-\delta\mu/\sigma\bigr]$.
This is only possible in the region where $r D_1-1 \leq -\mu/\sigma$.
From \eqref{Eq:dI/dtau1} we have
\begin{equation}  \label{Eq:I1}
\mathcal{I}(\hat u,\delta,|c|,2)
=\mathcal{I}_{\!1}(\hat u,\delta):=
\begin{cases}
\hat u\Bigl[ e^{\mu r}  + \frac{\sigma}{\mu}(e^{\mu r}- 1) \Bigr]
+ \frac{\sigma}{\mu}D_1\delta\Bigl[\frac{1}{\mu}(e^{\mu r}- 1) - r e^{\mu r} \Bigr], & \mu\ne0,\\
\hat u  +\sigma r\hat u -\frac{\sigma D_1  r^2}{2}\delta, & \mu=0.
\end{cases}
\end{equation}
If $\tfrac{\delta+\hat u}{D_1\delta}<r$ then $\hat u$ has another upper bound $\hat u <(rD_1-1)\delta$
so $\hat u \in \bigl[-\delta,\min\bigl\{(rD_1-1)\delta,-\tfrac{\mu}{\sigma}\delta\bigr\}\bigr]$.
Since the integration is broken down into two parts in this case we label the expression we derive
as $\mathcal{I}_{\!2}$, and from \eqref{Eq:dI/dtau2} we have
\begin{equation} \label{Eq:I2}
\mathcal{I}(\hat u,\delta,|c|,2)
=\mathcal{I}_{\!2}(\hat u,\delta):=
\begin{cases}
\displaystyle \hat u \Bigl(e^{\mu r}-\tfrac{\sigma}{\mu}\Bigr)
+ \tfrac{\sigma}{\mu}\delta\Bigl[\tfrac{D_1}{\mu}\Bigl(e^{\;\mu\tfrac{\delta+\hat u}{D_1\delta}}- 1\Bigr) - e^{\mu r} \Bigr], & \mu\ne0, \\
\displaystyle \hat u  -\sigma\delta r +\tfrac{\sigma}{2D_1\delta}(\delta+\hat u)^2, & \mu=0.
\end{cases}
\end{equation}
The main differences between the expressions for $\mathcal{I}(\hat u,\delta,c,2)$ and $\mathcal{I}(\hat u,\delta,|c|,2)$ are that
the former involve $r_+=a+c\delta$, and the latter use $r=a+|c|\delta$ as well as being subject to different
restrictions on the values of $\hat u$ for which they apply.
In \eqref{Eq:dI/dtau1},\eqref{Eq:dI/dtau2},\eqref{Eq:I1} and \eqref{Eq:I2} the $\mu=0$ expressions equal the $\mu\to0$ limit
of the $\mu\ne0$ expressions. Results for $\mu=0$ thus follow from those for $\mu\ne0$, and so we do
not treat these cases separately below.

\begin{thm}\label{Thm:StabilityDDE:simplify}
Let $a>0$, $c\ne 0$, $\sigma\leq\mu$ and $\sigma<-\mu$. Let $\delta\in(0,|a/c|)$.
If $P(\delta,c,2)<\delta$ then
\begin{equation} \label{eq:Pdeltac2}
P(\delta,c,2) = \left\{ \begin{array}{ll}
   \mathcal{I}_{\!1}(-\delta\mu/\sigma,\delta), & \text{if}\quad r D_1 - 1 \leq  -\mu/\sigma,\\
   \mathcal{I}_{\!2}(-\delta\mu/\sigma,\delta), & \text{if}\quad r D_1 - 1 >  -\mu/\sigma,
 \end{array} \right.
 \end{equation}
where $\mathcal{I}_{\!1}$ is defined by \eqref{Eq:I1} and $\mathcal{I}_{\!2}$ is defined by \eqref{Eq:I2}.
\end{thm}

\begin{proof}
See Appendix~\ref{Appendix:Explicit}.
\end{proof}

We will not state an explicit expression for $P(\delta,c,3)$.
When needed, this can determined by evaluating \eqref{Eq:P} numerically for $k=3$.

We now prove four lemmas which will be needed for the proof of Theorem~\ref{Thm:StabilityDDE:Razumikhin_k=123}
where we show  asymptotic stability in the set
$\bigl\{(\mu,\sigma) : P(1,0,k)<1\bigr\}$.

\begin{lem}\label{Lem:dI/dtau1}
Let $a>0$, $c\ne 0$, $\sigma\leqslant\mu$ and $\sigma<-\mu$. Let $\delta\in \bigl(0,|a/c|\bigr)$ and $\hat u \in\bigl[-\delta,-\delta\mu/\sigma\bigr]$ be fixed. Then $\mathcal I(\hat u,\delta,c,1)$
decreases with decreasing $r_+$.
\end{lem}

\begin{proof}
For $r_+>0$,
$$\frac{\partial}{\partial r_+}\biggl(\hat ue^{\mu r_+}
+ \sigma \int_{-r_+}^{0}{e^{-\mu \theta}\eta_{(1)}(\theta)d\theta}\biggr)
=e^{\mu r_+}\bigl[\mu \hat u +\sigma \eta_{(1)}(-r_+)\bigr]
=e^{\mu r_+}\bigl[\mu\hat{u} - \sigma\delta \bigr] > 0,$$
since $\sigma<-\mu$ and $\hat u\in\bigl[-\delta,-\delta\mu/\sigma\bigr]$.
\end{proof}

\begin{lem}\label{Lem:dI/dtau23}
For $k=2$ or $3$, let $a>0$, $c\ne 0$, $\sigma\leq\mu$ and $\sigma<-\mu$. Let $\delta\in(0,|a/c|)$
and $\hat u \in[-\delta,-\mu\delta/\sigma]$ be fixed.
Let $\mathcal{I}(r_+)$ be the expression for $\mathcal{I}(\hat u,\delta,c,k)$ as a function of only $r_+$;
\begin{equation} \label{Eq:dI/dtau23}
\mathcal{I}(r_+)
=\hat ue^{\mu r_+}
+ \sigma\! \int_{-r_+}^{0} \! {e^{-\mu \theta} \eta_{(k)}(\theta)d\theta}, \qquad
\tfrac{\partial}{\partial r_+}\mathcal{I}(r_+)
=e^{\mu r_+}\bigl[\mu \hat u +\sigma \eta_{(k)}(-r_+)\bigr].
\end{equation}
\begin{enumerate}[(A)]
\item If $\mu \leq 0$, then $\frac{\partial}{\partial r_+}\mathcal{I}(r_+) > 0$.
\item If $\mu>0$ and $\eta_{(k)}(-r_+)\leq0$, then $\frac{\partial}{\partial r_+}\mathcal{I}(r_+) > 0$.
\item If $\frac{\partial}{\partial r_+}\mathcal{I}(r_+) \leq 0$, then $\mu>0$, $\eta_{(k)}(-r_+)>0$, and $\mathcal{I}(\hat u,\delta,c,k)<\delta$.
\end{enumerate}
\end{lem}

\begin{proof}
Parts (A) and (B) are easy to show. Let $\frac{\partial}{\partial r_+}\mathcal{I}(r_+) \leq 0$. From the
first two cases, this is only possible if $\mu>0$ and $\eta_{(k)}(-r_+)>0$.
Consider $k=2$ first.
Since $\eta_{(2)}(-r_+)>0\ne -\delta$ we are in the case $\frac{\delta+\hat u}{D_1\delta}>r_+$.
Thus $\mathcal{I}(\hat u, \delta,c,2)$ is given by \eqref{Eq:dI/dtau1}.
In this case, $\eta_{(2)}(-r_+)= \hat u - D_1\delta r_+>0$. From $\frac{\partial}{\partial r_+}\mathcal{I}(r_+) \leq 0$,
we get $\hat u\leq  -\frac{\sigma}{\mu}\eta_{(2)}(-r_+) = -\frac{\sigma}{\mu}(\hat u - D_1\delta r_+)$ which we can rearrange to get 
$-(\mu+\sigma)\hat u\geq-\sigma D_1\delta r_+$. Using this in \eqref{Eq:dI/dtau1} yields,
\begin{align*}
\mathcal{I}(\hat u, \delta,c,2)
&\leqslant \hat u\Bigl[ e^{\mu r_+}  + \frac{\sigma}{\mu}(e^{\mu r_+}- 1)\Bigr]
+ \delta\frac{\sigma D_1}{\mu^2}(e^{\mu r_+}  - 1) -  \frac{(\mu+\sigma)}{\mu}\hat u e^{\mu r_+},\\
&=-\hat u\frac{\sigma}{\mu}  + \delta \frac{\sigma D_1}{\mu^2 }(e^{\mu r_+} - 1), \\
& \leq \delta  + \delta \frac{\sigma D_1}{\mu^2 }(e^{\mu r_+} - 1),
\quad \text{since $\hat u \leq -\frac{\mu}{\sigma}\delta$}, \\
& <\delta.
\end{align*}

For $k=3$,
from $D_1^2/D_2\leqslant 1$ it follows that
$$\bar \eta_{(3)}(D_1/D_2) = -\delta+\frac{\delta D_1^2}{2D_2} \leqslant-\delta+\frac{\delta}{2}=-\frac{\delta}{2}<0.$$
Since $\eta_{(3)}$ is an increasing function and we require $\eta_{(3)}(-r_+)>0$,
then $-r_+ +\theta_\text{shift}>\frac{D_1}{D_2}$. Thus $\theta+\theta_\text{shift} >D_1/D_2$ for all $\theta\in[-r_+,0]$.
By the definition of $\eta_{(3)}$, in this case $\eta_{(3)}(\theta)=\eta_{(2)}(\theta)$ for $\theta\in[-r_+,0]$
and $\mathcal{I}(\hat u, \delta,c,3)=\mathcal{I}(\hat u, \delta,c,2)<\delta$.
\end{proof}

\begin{lem}\label{Lem:StabilityDDE:negativec123}
For $k=1$, $2$ or $3$, let $a>0$, $c\ne 0$, $\sigma\leqslant\mu$ and $\sigma<-\mu$.
Let $\delta\in \bigl(0,|a/c|\bigr)$. Then
\[\left\{(\mu,\sigma)\, : \, P(\delta,c,k)=\sup\limits_{\hat u \in[-\delta,-\delta\mu/\sigma]}\mathcal{I}(\hat u,\delta,|c|,k)<\delta\right\}
\subseteq \left\{(\mu,\sigma)\; : \;\sup\limits_{\hat u \in[-\delta,-\delta\mu/\sigma]}\mathcal{I}(\hat u,\delta,-|c|,k)<\delta\right\}. \]
\end{lem}

\begin{proof}
Recall from Section~\ref{Sec:ModelEquation}
that changing the sign of $c$ in $\mathcal{I}(\hat u,\delta,c,k) $ only changes the value of $r_+=a+c\delta$.
For $k=1$ the result follows from Lemma~\ref{Lem:dI/dtau1}.

For $k=2$ or $3$, let $(\mu,\sigma)\in \bigl\{P(\delta,c,k)<\delta\bigr\}$ and
$\hat u\in[-\delta,-\delta\mu/\sigma]$.
Recall that $\sigma<0$, while
$\eta_{(k)}(\theta)$ is a nondecreasing function in $\theta$.
There are two cases to consider:
\begin{enumerate}[(i)]
\item If $\mu\hat u+\sigma\eta_{(k)}(-(a-|c|\delta))\leq 0$ then
                  by Lemma~\ref{Lem:dI/dtau23}(C), $\mathcal{I}(\hat u,\delta,-|c|,k)<\delta$.
\item If $\mu\hat u+\sigma\eta_{(k)}(-(a+|c|\delta))\geq \mu\hat u+\sigma\eta_{(k)}(-(a-|c|\delta))>0$, then $\mu\hat u+\sigma\eta_{(k)}(-\tau))>0$ for all
        $\tau\in(a-|c|\delta,a+|c|\delta)$.
        By equation \eqref{Eq:dI/dtau23},
        the expression for $\mathcal{I}(r_+)$ is
        increasing over this interval and
        thus, $\mathcal{I}(\hat u,\delta,-|c|,k)\leq\mathcal{I}(\hat u,\delta,|c|,k) \leq P(\delta,c,k) <\delta$.
\end{enumerate}
Thus $\mathcal{I}(\hat u,\delta,-|c|,k)<\delta$ and the result follows.
\end{proof}

\begin{lem}\label{Lem:StabilityDDE:P123}
For $k=1$, $2$ or $3$, let $a>0$, $c\ne 0$, $\sigma\leqslant\mu$ and $\sigma<-\mu$. If $0<\delta_{*}\leqslant \delta_{**}<|a/c|$ then
$$\bigl\{(\mu,\sigma)\, : \, P(\delta_{**},c,k)<\delta_{**}\bigr\}
\subseteq \bigl\{(\mu,\sigma)\, : \, P(\delta_*,c,k)<\delta_{*}\bigr\}.$$
\end{lem}

\begin{proof}
Increasing $\delta$ increases $r=a+|c|\delta$ which is the only source of nonlinearity in $\delta$ in the first expression
\eqref{eq:Pdeltac1} for $P(\delta,c,1)$. Thus for $\mu\ne0$
\begin{equation}  \label{Eq:dI/ddelta1}
\frac{\partial}{\partial\delta}\biggl(\frac{P(\delta,c,1)}{\delta}\biggr)
=-e^{\mu r}\biggl(\frac{\mu}{\sigma}+\frac{\sigma}{\mu}\biggr)\frac{\partial}{\partial\delta}(\mu r)
=\mu|c|\biggl(\frac{\mu^2+\sigma^2}{-\sigma\mu}\biggr)e^{\mu r}>0.
\end{equation}
Positivity also follows trivially from \eqref{eq:Pdeltac1} when $\mu=0$. The result follows for $k=1$.

For $k=2$ or $3$,
consider $\mathcal{I}(s\delta,\delta,|c|,k)/\delta$ and note that
$r$, $D_1$ and $D_2$ are the only terms in the expression that depend on $\delta$,
and that increasing $\delta$ increases the value of each of those terms. Thus
$$\frac{\partial}{\partial\delta}\left(\frac{\mathcal{I}(s\delta,\delta,|c|,k)}{\delta}\right)
= \frac{\partial}{\partial r}\left(\frac{\mathcal{I}(s\delta,\delta,|c|,k)}{\delta}\right)|c|
+ \sum_{j=1}^{k-1}\frac{\partial}{\partial D_j}\left(\frac{\mathcal{I}(s\delta,\delta,|c|,k)}{\delta}\right)
\frac{\partial D_j}{\partial \delta}.$$ 
We focus on the first term on the right-hand side, since all the remaining terms are positive.
From \eqref{Eq:Ick} we can write
$$\frac{\partial}{\partial r}\left(\frac{\mathcal{I}(s\delta,\delta,|c|,k)}{\delta}\right)
=e^{\mu r}\Bigl[\mu s + \tfrac{\sigma}{\delta}\eta_{(k)}(-r)\Big]
=e^{\mu r}\Bigl[\mu s + \tfrac{\sigma}{\delta}\eta_{(k)}(-(a+|c|\delta)\Big].$$
Let $r^{*}=a+|c|\delta_*$, $r^{**}=a+|c|\delta_{**}$ and
$(\mu,\sigma)\in \bigl\{P(\delta_{**},c,k)<\delta_{**}\bigr\}$.
Let $s\in[-1,-\mu/\sigma]$ and use the notation $\eta_{(k)}(\delta,\theta)$ to denote the function $\eta_{(k)}$
as a function of both $\theta$ and $\delta$. Note that
$\eta_{(k)}(\delta,-(a+|c|\delta)) / \delta$ is a nonincreasing function of $\delta$.
Consider the following cases:
\begin{enumerate}[(i)]
\item
If $\mu s +\sigma\eta_{(k)}(\delta_*,-r^*) / \delta_* \leq0$ then
by Lemma~\ref{Lem:dI/dtau23}(C), $\mathcal{I}(s\delta_*,\delta_*,|c|,k)<\delta_*$.
\item
If $\mu s +\sigma\eta_{(k)}(\delta_{**}, -r^{**}) /\delta_{**}\geq  \mu s +\sigma\eta_{(k)}(\delta_*,-r^*) / \delta_* >0$
then $\frac{\partial}{\partial r}\Bigl(\frac{\mathcal{I}(s\delta,\delta,|c|,k)}{\delta}\Bigr)\geq0$
for $\delta\in[\delta_*,\delta_{**}]$. Thus,
$\frac{\partial}{\partial \delta}\Bigl(\frac{\mathcal{I}(s\delta,\delta,|c|,k)}{\delta}\Bigr)\geq0$
for $\delta\in[\delta_*,\delta_{**}]$ and,
\[\frac{\mathcal{I}(s\delta_*,\delta_{*},|c|,k)}{\delta_{*}}\leq\frac{\mathcal{I}(s\delta_{**},\delta_{**},|c|,k)}{\delta_{**}} \leq \frac{P(\delta_{**},c,k)}{\delta_{**}}<1.\]
\end{enumerate}
Cases (i) and (ii) both yield $\mathcal{I}(s\delta_*,\delta_*,|c|,k)<\delta_*$.
Since this holds for all $s\in[-1,-\mu/\sigma]$, $P(\delta_*,|c|,k)<\delta_*$ follows.
\end{proof}


With these lemmas we can prove our main result.

\begin{thm}[Asymptotic stability for \eqref{Eq:1Delay} using $\mathcal{E}_{(k)}(\delta,x)$]
\label{Thm:StabilityDDE:Razumikhin_k=123}
For $k=1$, $2$ or $3$,
let $a>0$, $c\ne 0$, and $(\mu,\sigma) \in \{P(1,0,k) <1\}$ where $P(\delta,c,k)$ is defined by \eqref{Eq:P}.
Then $(\mu,\sigma) \in \{P(\delta_1,c,k) <\delta_1\}$ for some $\delta_1\in\bigl(0,|a/c|\bigr)$. Furthermore, for $\delta\in(0,\delta_1]$ let $\delta_2=\delta e^{-k(|\mu|+|\sigma|)(a+|c|\delta)}$
and $|\varphi(t)|< \delta_2 $ for all $t\in[-a-|c|\delta,0]$, then the solution to \eqref{Eq:1Delay}
satisfies $|u(t)|\leqslant \delta$ for all $t\geqslant 0$ and $\lim_{t\to\infty}u(t)=0$.
\end{thm}

\begin{proof}
For this proof define
$$J=\bigcup_{\delta\in(0,|a/c|)} \Bigl\{P(\delta,c,k)<\delta\Bigr\}.$$
First we show that $J=\{P(1,0,k) <1\}$.
When $c=0$ it is seen that $\mathcal{I}(s\delta,\delta,0,k)/\delta$ is independent of $\delta$ for $k=1$, $2$ or $3$.
From this it follows that $P(\delta,0,k)/\delta=P(1,0,k)$.
Moreover, for all $c$, when $\delta\to 0$ then $r\to a$, and
$\mathcal{I}(s\delta,\delta,|c|,k)/\delta\to\mathcal{I}(s\delta,\delta,0,k)/\delta$,
since $c$ only appears multiplied by $\delta$ in these expressions.
Thus $P(\delta,c,k)/\delta\to P(1,0,k)$ as $\delta\to 0$.
Because of this and Lemma~\ref{Lem:StabilityDDE:P123}, $J=\{P(1,0,k) <1\}$.

Let $(\mu,\sigma)\in\{P(1,0,k) <1\}$. The existence of $\delta_1$ such
that $(\mu,\sigma) \in \{P(\delta_1,c,k) <\delta_1\}$ follows from the above discussion.
It also follows that $(\mu,\sigma) \in \{P(\delta,c,k) <\delta\}$ for all $\delta\in(0,\delta_1]$.

Let $\delta\in(0,\delta_1]$. Consider the auxiliary ODE \eqref{Eq:AuxiliaryODESDD}.
For all $\eta\in\mathcal{E}_{(k)}(\delta,\delta)$ it follows from
Lemmas~\ref{Lem:v<P} and~\ref{Lem:StabilityDDE:negativec123} that
$v(0)\leq \sup_{\hat u \in[-\delta,-\delta\mu/\sigma]}\mathcal{I}(\hat u,\delta,c,k)
\leqslant P(\delta,c,k)<\delta$.
Similarly, for all $\eta\in\mathcal{E}_{(k)}(\delta,-\delta)$ we obtain
$v(0)\geq-\sup_{\hat u \in[-\delta,-\delta\mu/\sigma]}\mathcal{I}(\hat u,\delta,-c,k)
\geqslant -P(\delta,c,k)>-\delta$.
Thus  \eqref{odecond} holds for all $\eta\in\mathcal{E}_{(k)}(\delta,x)$
or any $|x|=\delta$. This is true for all $\delta\in(0,\delta_1]$.
Since $E_{(k)}(\delta,x)\subseteq\mathcal{E}_{(k)}(\delta,x)$,
applying Theorem~\ref{Thm:Technique} completes the proof.
\end{proof}

\begin{figure}[t]
\begin{center}
\subfigure[$\{P(1,0,1)<1\}$]{\includegraphics[scale=1]{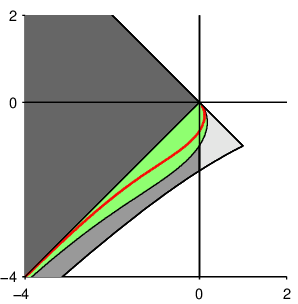}\label{Fig:StabilityExample_k=1}}
\put(-145,115){{\footnotesize $\sigma$}}
\put(-25,1){{\footnotesize $\mu$}}
\hspace{3em}
\subfigure[$\{P(1,0,2)<1\}$]{\includegraphics[scale=1]{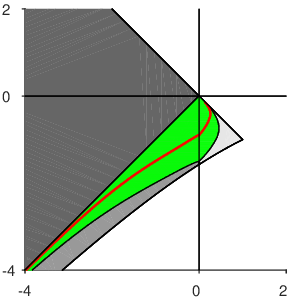}\label{Fig:StabilityExample_k=2}}
\put(-140,115){{\footnotesize $\sigma$}}
\put(-20,1){{\footnotesize $\mu$}}
\FigureCaption{For (a) $k=1$ and (b) $k=2$, the set $\{P(1,0,k)<1\}$ is shaded green and the
boundary of $\{P(1/2,1,k)<1/2\}$ is drawn in red. If $(\mu,\sigma)\in\{P(1,0,k)<1\}$ then the zero solution to \eqref{Eq:1Delay} is
asymptotically stable by Theorem~\ref{Thm:StabilityDDE:Razumikhin_k=123}.
As $\delta \to 0$ the proof of Theorem~\ref{Thm:StabilityDDE:Razumikhin_k=123} shows
that $\{P(\delta,c,k)<\delta\}$ converges to $\{P(1,0,k)<1\}$.}
\label{Fig:StabilityExample_k=12} 
\end{center}
\end{figure}

For $(\mu,\sigma)$ in the part of $\wedg\cup\cusp$ for which $P(1,0,k)<1$,
equation
\eqref{odecond} is satisfied for all $\eta\in\mathcal{E}_{(k)}(\delta,x)$ and hence
Theorem~\ref{Thm:StabilityDDE:Razumikhin_k=123} establishes asymptotic stability
These parameter sets are shown
in Figure~\ref{Fig:StabilityExample_k=12} for $k=1$ and $k=2$.
The stability region $\{P(1,0,1)<1\}$ shown in Figure~\ref{Fig:StabilityExample_k=12}(a)
for the DDE \eqref{Eq:1Delay} comprises a relatively small
part of $\wedg\cup\cusp$, because it is derived by requiring that \eqref{odecond} holds for
all $\eta\in\mathcal{E}_{(1)}(\delta,x)$. But $\mathcal{E}_{(1)}(\delta,x)$ is a very large set, with the
main restrictions on $\eta$ being that it is merely piecewise continuous with $\|\eta\|\leq\delta$.

We obtain a larger stability region by increasing $k$. This is seen in Figure~\ref{Fig:StabilityExample_k=12}(b)
where $P(1,0,2)<1$ ensures that \eqref{odecond} is satisfied for all $\eta\in\mathcal{E}_{(2)}(\delta,x)$
results in a significantly larger stability region than seen in Figure~\ref{Fig:StabilityExample_k=12}(a).
Since $\mathcal{E}_{(2)}(\delta,x)\subset\mathcal{E}_{(1)}(\delta,x)$, with all functions $\eta\in\mathcal{E}_{(2)}(\delta,x)$
satisfying the derivative bound $|\eta'(\theta)|\leq D_1\delta$, the set $\mathcal{E}_{(2)}(\delta,x)$ is smaller than
$\mathcal{E}_{(1)}(\delta,x)$ and it is possible to satisfy \eqref{odecond} over a larger region of $(\mu,\sigma)$
parameter space. We will compare the sizes of the stability regions $\{P(1,0,k)<1\}$ for different $k$ in
Section~\ref{Sec:StabilityDDE:Measurements}.

The boundary of $\{P(1/2,1,k)<1/2\}$ is also shown in Figure~\ref{Fig:StabilityExample_k=12} for $k=1$ and $k=2$.
As $\delta \to 0$ the sets $\{P(\delta,c,k)<\delta\}$ converge to the set $\{P(1,0,k)<1\}$, and
for $(\mu,\sigma)\in\{P(1,0,k)<1\}$
the inequality $P(\delta,c,k)<\delta$ can be used to determine the largest $\delta_1$ and hence the largest
$\delta_2$ for which Theorem~\ref{Thm:StabilityDDE:Razumikhin_k=123} applies.
This determines a ball which is contained in the basin of attraction of the steady state, and in
Section~\ref{Sec:Basins} we consider how the size of this lower bound on the basin of attraction varies with $k$.


%% file: Measurements.tex
In this section we compare the sets in which we can establish asymptotic stability of the
steady state of the state-dependent DDE \eqref{Eq:1Delay} using Lyapunov-Razumikhin techniques.
In Sections~\ref{Sec:Stability:k=123} we showed asymptotic stability
for $(\mu,\sigma)\in\cone$, and for $(\mu,\sigma)$
in the parts of the cusp $\cusp$ and wedge $\wedg$ for which $P(1,0,k)<1$ for $k=1,2,3$.
Measurements of these sets and the exact stability region $\Sigma_\star$ are presented in
Tables~\ref{Table:NumericalTests1}--\ref{Table:NumericalTests3}, and they are
illustrated 
in Figure~\ref{Fig:StabilityDDESummary}.

To compute these stability regions, from
Theorems~\ref{Thm:StabilityDDE:Razumikhin_k=123}
we need to compute $P(1,0,k)$ in the limiting case $c=0$, $\delta=1$.
This was done in MATLAB \cite{Matlab}. For $k=1$ and $2$ we have exact expressions for  $P(1,0,k)$ given by
\eqref{eq:Pdeltac1} and  \eqref{eq:Pdeltac2}.
Noting that $\sigma<0$ in $\wedg\cup\cusp$, from
\eqref{eq:Pdeltac1} we find that $(\mu,\sigma)$ satisfies $P(1,0,1)<1$ when
\begin{equation} \label{eq:P101<1}
\tfrac{\sigma^2}{\mu}(1-e^{\mu a})-\sigma-\mu e^{\mu a}>0.
\end{equation}
The boundary of $\{P(1,0,1)<1\}$ is defined by equality in \eqref{eq:P101<1}.

For $k=3$
the value of $P(1,0,3)$ was calculated by maximizing the function $\mathcal{I}(\hat u,1,0,k)$
over $\hat u\in [-1,-\mu/\sigma]$ using the MATLAB \verb#fminbnd# function.

The boundary of
$\{P(1,0,k)<1\}$ is then found by fixing one of $\mu$ or $\sigma$ and
using the \verb#fzero# function to find the value of the other one which solves $P(1,0,k)-1=0$
(except in the case $k=1$ where for given $\mu$, applying the quadratic formula to \eqref{eq:P101<1} determines $\sigma$).
The largest value of $\mu$ for each region (shown in Table~\ref{Table:NumericalTests3})
is then found by regarding the $\mu$ that solves $P(1,0,k)=1$ as a function of $\sigma$ and
using \verb#fminsearch# to find the $\sigma$ that maximises $\mu$. The boundary of the
full stability domain $\Sigma_\star$,
found by linearization, is given by Definition~\ref{Defn:StabilityRegion}.

\begin{figure}[t]
\subfigure[$\cone\cup\{P(1,0,1)<1\}$]{\includegraphics[scale=1]{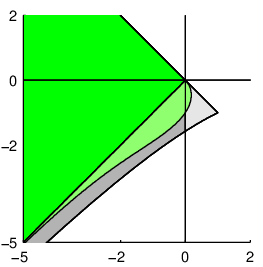}}
\subfigure[$\cone\cup\{P(1,0,2)<1\}$]{\includegraphics[scale=1]{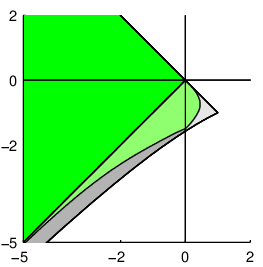}\label{Fig:StabilityRegion_k=22}}
\subfigure[$\cone\cup\{P(1,0,3)<1\}$]{\includegraphics[scale=1]{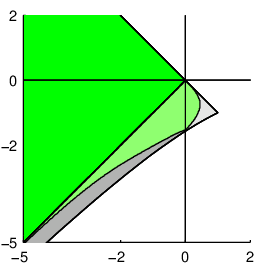}\label{Fig:StabilityRegion_k=3}}
\put(-128,102){{\footnotesize $\sigma$}}
\put(-255,102){{\footnotesize $\sigma$}}
\put(-22,1){{\footnotesize $\mu$}}
\put(-382,102){{\footnotesize $\sigma$}}
\put(-282,1){{\footnotesize $\mu$}}
\put(-150,1){{\footnotesize $\mu$}}
\caption{Stability regions found using
Theorems~\ref{Thm:localcone} and~\ref{Thm:StabilityDDE:Razumikhin_k=123}
shaded in green. Parameter pairs $(\mu,\sigma)\in\Sigma_\star$
which are contained in the wedge $\wedg$ or cusp $\cusp$ but not in $\{P(1,0,k)\}$ are shaded dark or light
grey respectively.}  \label{Fig:StabilityDDESummary}
\end{figure}

\begin{table}[t]
  \centering
 \begin{tabular}{|c|l|l|l|}
  \hline
   Region   & $\sigma$ at $\mu=-5$ & $\sigma$ at $\mu=-2$ & $\sigma$ at $\mu=0$\\
  \hline
  $\left\{P(1,0,1)<1\right\}$ & $-5.0673855$ & $-2.5578041$ &  $-1$\\
  \hline
  $\left\{P(1,0,2)<1\right\}$ & $-5.0676090$ & $-2.5875409$ & $-1.5=-3/2$\\
  \hline
  $\left\{P(1,0,3)<1\right\}$ & $-5.0678325$ & $-2.6127906$ & $-1.5416667=-37/24$\\
  \hline
  $\Sigma_\star$  & $-5.6605586$ & $-3.0396051$ & $-1.5707963=-\pi/2$ \\
  \hline
\end{tabular}
 \TableCaption{Boundaries of the stability regions:
 Values of $\sigma$ for fixed $\mu$ with $a=1$.} \label{Table:NumericalTests1}
\end{table}

\begin{table}[t]
  \centering
\begin{tabular}{|c|l|l|l|}
  \hline
   Region     & $\mu$ at $\sigma=-5$ &  $\mu$ at $\sigma=-2$ & $\mu$ at $\sigma=-1$ \\
  \hline
  $\left\{P(1,0,1)<1\right\}$ & $-4.9286634$ & $-1.16401463$ & $0$ \\
  \hline
  $\left\{P(1,0,2)<1\right\}$ & $-4.9283941$ & $-1.0040856$ & $0.38774807$ \\
  \hline
  $\left\{P(1,0,3)<1\right\}$ & $-4.9281247$ & $-0.93607885$ & $0.39925645$ \\
  \hline
  $\Sigma_\star$ & $-4.2734224$ & $-0.63804505$  & $1$ \\
  \hline
\end{tabular}
  \TableCaption{\label{Table:NumericalTests2}
  Boundaries of the stability regions: Values of $\mu$ for fixed $\sigma$ with $a=1$.}
\end{table}

\begin{table}[t]
  \centering
\begin{tabular}{|c|l|l|}
  \hline
    Region   & Supremum of $\mu$ & Corresponding value of $\sigma$ \\
  \hline
  $\left\{P(1,0,1)<1\right\}$ & $0.18822641=\ln((1+\sqrt{2})/2))$ & $-0.45439453$ \\
  \hline
  $\left\{P(1,0,2)<1\right\}$ & $0.45697166$ & $-0.73935547$ \\
  \hline
  $\left\{P(1,0,3)<1\right\}$ & $0.45700462$ & $-0.74059482$ \\
  \hline
  $\Sigma_\star$ & $1$ & $-1$  \\
  \hline
\end{tabular}
  \TableCaption{The values of $\mu$ and $\sigma$ at the rightmost boundary point of each stability region with $a=1$.} \label{Table:NumericalTests3}
\end{table}

Since by Theorem~\ref{Thm:StabilityDDE:Razumikhin_k=123},
at least for $k\leq 3$,
the Lyapunov-Razumikhin stability
regions are given by $P(1,0,k)<1$ irrespective
of the value of $c$, we obtain the same regions in the constant $c=0$ and variable
$c\ne0$ delay cases. This is consistent with the linearization theory of Gy\"ori and Hartung \cite{GyoriHartung:1}
who showed that $\Sigma_\star$ is the exponential stability region for both $c=0$ and $c\ne0$.

When $\mu=0$ the DDE \eqref{Eq:1Delay} becomes
\begin{equation} \label{Eq:1Delmu0}
\dot u(t)=\sigma u(t-a-cu(t))
\end{equation}
and the intervals of $\sigma$ values in the stability regions,
(shown in the last column of Table~\ref{Table:NumericalTests1})
can be found exactly.
From \eqref{eq:Pdeltac1} we have $P(1,0,1)=-\sigma a$ when $\mu=0$ which implies $\sigma\in(-1/a,0)$
for $P(1,0,1)<1$. Similarly, when $\sigma<-1/a$
from \eqref{eq:Pdeltac2} we have $P(1,0,2)=-\sigma a-1/2$ and hence $\sigma=-3/(2a)$
on the boundary of $\{P(1,0,2)<1\}$.
Magpantay \cite{Magpantay:1} shows that for $\mu=0$ we require $\sigma\in(-37/(24a),0)$ for $P(1,0,3)<1$.
For the constant delay case of \eqref{Eq:1Delmu0}, with $c=0$, Barnea \cite{Barnea:1} showed Lyapunov
stability for $\sigma\in[-3/(2a),0]$, by applying Lyapunov-Razumikhin techniques with $k=2$. The stability
bound $\sigma\geq-37/(24a)$ seems not to have been derived before for the constant delay case of \eqref{Eq:1Delmu0},
but is well-known for Wright's equation \cite{Wright:1} which is a nonlinear constant delay DDE whose
linear part corresponds to \eqref{Eq:1Delmu0} with $c=0$ (see \cite{Liz2003309,Wright:1}).

When $\mu=0$,
the boundary of $\{P(1,0,k)<1\}$ seems to converge rapidly to $-\pi/(2a)$,
the boundary of $\Sigma_\star$, as $k\to\infty$,
suggesting that the full stability interval can be recovered. Indeed, for constant delay with $\mu=c=0$
Krisztin \cite{Krisztin:1} showed Lyapunov stability for $\sigma\in(-\pi/(2a),0]$ by considering $k\to\infty$.


Barnea \cite{Barnea:1} and Myshkis \cite{Myshkis:1} also applied Razumikhin techniques to establish Lyapunov stability for \eqref{Eq:1Delay} in the case of constant delay ($c=0$) with $\mu\ne0$. The regions in which they claim stability are shown
in Figure~\ref{fig:BMLyap}. The region found by Myshkis \cite{Myshkis:1} has $\sigma\geq-1/a$ and $\mu\leq\sigma(a\sigma+1)/(a\sigma-1)$, which for $a=1$ always has $\mu\leq3-2\sqrt{2}$ and is contained in $\{P(1,0,1)<1\}$.

Barnea \cite{Barnea:1} claimed that
the Lyapunov stability region of \eqref{Eq:1Delay} with $c=0$ contains the
region $\mathcal{X}_2=\bigl\{(\mu,\sigma):0\leqslant s^* \leqslant a, \; P<1\bigr\}$  where
\[s^*=-\frac{e^{\mu a}}{\sigma}, \qquad P=\frac{\sigma(\mu+\sigma)}{\mu^2}\left[e^{\mu s^*}-\frac{\sigma}{\mu+\sigma}\right].\]
We show this region in Figure~\ref{Fig:X2}, but
Barnea did not actually graph $\mathcal{X}_2$ or give its derivation in \cite{Barnea:1}.
He noted that setting $P=1$ and letting $\mu\to0$ yields that the point $\sigma = -3/2a$
is a boundary of $\mathcal{X}_2$ on the $\sigma$-axis. We observe that setting $s^* =0$ and $\mu\to0$
yields $\sigma = -1/a$ as the other boundary of $\mathcal{X}_2$ on the $\sigma$-axis.
Thus the region $\mathcal{X}_2$ does not include the whole interval $\sigma\in(-3/2a,0]$ on the
$\sigma$-axis which Barnea had proven to be Lyapunov stable in the $\mu=0$ case in the same paper \cite{Barnea:1}.
Barnea's stability region $\mathcal{X}_2$ is hence incomplete. Although the $\eta_{(2)}$ function
used by Barnea to show Lyapunov stability corresponds to \eqref{Eq:eta2} with $c=0$,
it appears that Barnea performed his integration assuming that
$\frac{\delta+\hat u}{D_1\delta}\leq r_+$ in all cases.
The case when $\frac{\delta+\hat u}{D_1\delta}>r_+$ occurs in the $\mu=c=0$ case
(as well as the general case $c\ne0$, $\mu\ne0$ considered in \eqref{Eq:I1},\eqref{Eq:I2}).
Omitting this case results in the incorrect stability region $\mathcal{X}_2$.
The correct region is $(\mu,\sigma)\in\bigl\{P(1,0,2)<1\bigr\}$ as illustrated in Figure~\ref{Fig:StabilityRegion_k=22}.
Moreover within this region we show the stronger property of asymptotic stability for both
the constant delay ($c=0$) and state-dependent delay ($c\ne 0$) cases.

\begin{figure}[t]
\begin{center}
\subfigure[]{\includegraphics[scale=1]{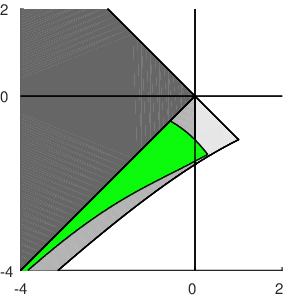}\label{Fig:X2}}
\put(-140,115){{\footnotesize $\sigma$}}
\put(-25,1){{\footnotesize $\mu$}}
\hspace{3em}
\subfigure[]{\includegraphics[scale=1]{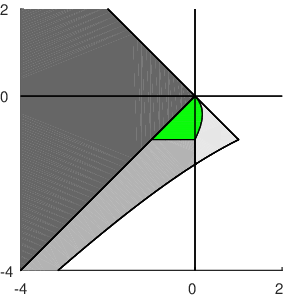}\label{Fig:M2}}
\put(-140,115){{\footnotesize $\sigma$}}
\put(-25,1){{\footnotesize $\mu$}}
\FigureCaption{(a) The set $\mathcal{X}_2\setminus\cone$ which is part of the stability region of \eqref{Eq:1Delay} for the constant delay case ($c=0$) according to Barnea \cite{Barnea:1}.
(b) The part of the stability domain outside $\cone$ for the same problem found by Myshkis \cite{Myshkis:1}.}
\label{fig:BMLyap}
\end{center}
\end{figure}



Tables~\ref{Table:NumericalTests1} and~\ref{Table:NumericalTests2} also show that for $\mu<0$ we can show asymptotic
stability in a larger part of $\Sigma_\star$ by increasing $k$. However when
$\mu \ll0$ the improvement in going from $k=1$ to $2$ to $3$ is very marginal and we can only show asymptotic stability
in a slice of the wedge $\wedg$ whose width appears to go to zero as $\mu\to-\infty$.
The problem here is that as $\mu\to-\infty$ the DDE \eqref{Eq:1Delay}
is singularly perturbed and can be written as the so-called saw-tooth equation
$$\epsilon\dot{u}(t)=u(t)+Ku(t-a-cu(t))$$
where $\epsilon=1/\mu$ and $K=\sigma/\mu$.
This DDE had been studied in detail in \cite{MalletParetNussbaum:4} and
for $K>1$ sufficiently large (corresponding to $(\mu,\sigma)$ outside $\Sigma_\star$) the steady state is
unstable, but there is an asymptotically stable slowly oscillating periodic solution.
This periodic solution, known as the sawtooth solution,
has unbounded gradient and a discontinuous profile in the singular limit.
For parameter values inside the wedge $\wedg$
the steady state is asymptotically stable, and for large and negative $\mu$ there are no periodic solutions but a slowly decaying sawtooth-like oscillation can occur. Lyapunov-Razumikhin techniques based on
bounding derivatives of solutions cannot perform well when those derivatives can be arbitrarily large. To
improve the results in this case it would be necessary to define different sets $\mathcal{E}_{(k)}^*(\delta,x)$
which take into account the structure of the oscillations and are hopefully much closer to
$E_{(k)}(\delta,x)$ than the sets $\mathcal{E}_{(k)}(\delta,x)$ that we use here.

For $\mu>0$ there is a significant improvement in the computed stability domain in going from $k=1$ to $k=2$
and a smaller improvement using $k=3$. The largest value of $\mu$ which satisfies $P(1,0,1)\leq1$
can be computed from \eqref{eq:P101<1} which is quadratic in $\sigma$. Then non-negativity of the discriminant
imposes the bound that $\mu<(1/a)\ln((1+\sqrt{2})/2)\approx0.1882/a$, as seen in
Table~\ref{Table:NumericalTests3}.

Although the parameter regions in which we can show asymptotic stability are independent of $c$, we will
see in Section~\ref{Sec:Basins}
that the basins of attraction do depend on $c$.

%% file: Basins.tex
Theorem~\ref{Thm:StabilityDDE:Razumikhin_k=123} shows that for $(\mu,\sigma)\in\wedg\cup\cusp$
the ball
\begin{equation} \label{Eq:basink}
\left\{\phi:\|\phi\|< \delta_2=\delta_1e^{-k(|\mu|+|\sigma|)(a+|c|\delta_1)}\right\}
\end{equation}
is contained in the basin of attraction of the steady-state of the state-dependent DDE \eqref{Eq:1Delay}
for $k=1$, $2$ and $3$. For fixed $\delta_1$ the radius of this ball gets smaller as $k$ increases, but the value of $\delta_1$ depends on $k$, $\mu$ and $\sigma$, and some work is required to determine the largest such ball
that is contained in the basin of attraction.
In \cite{MH:1}
we show that \eqref{Eq:basink} can be improved when $\mu<0$, so here we will consider
$(\mu,\sigma)\in\cusp$,
where $\sigma<0\leq\mu$.
Lemma~\ref{lem:1delbound} does not apply
when $\mu\geqslant 0$, so there is no \textit{a priori} bound on the solutions to \eqref{Eq:1Delay} in this case.
We present two examples which show that \eqref{Eq:1Delay} can have unbounded solutions when $\mu\geq0$, which also
shows that the steady-state is not globally asymptotically stable when $(\mu,\sigma)\in\cusp$ and gives an
upper bound on the largest ball contained in its basin of attraction. For simplicity of exposition we suppose
$c>0$ in this section, but the results can easily be extended to $c<0$. We first consider $\mu=0$.



\newcommand{\dUB}{\delta^{*}}

\begin{figure}[t]
\begin{center}
\subfigure[{$\delta_1$ when $\mu=0$, $\sigma\in[-\pi/2,0]$}]{\includegraphics[scale=1]{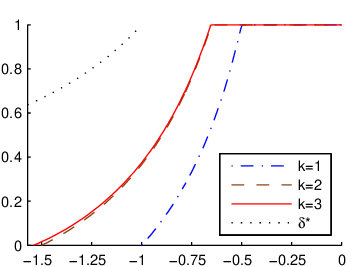}}
\put(-21,0){{\footnotesize $\sigma$}}
\put(-174,104){{\footnotesize $\delta$}}
\mbox{}\hspace{2em}\mbox{}
\subfigure[{$\delta_2$ when $\mu=0$, $\sigma\in[-\pi/2,0]$}]{\includegraphics[scale=1]{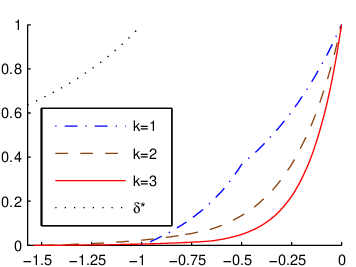}}
\put(-21,0){{\footnotesize $\sigma$}}
\put(-174,104){{\footnotesize $\delta$}}
\caption{For fixed $a=c=1$, $\mu=0$ and $\sigma\in[-\pi/2,0]$, (a) supremum of $\delta_1\in(0,a/c]$ such that $(\mu,\sigma)\in\{P(\delta_1,c,k)<\delta_1\}$
with $k=1$, $2$, $3$, and,
(b) $\delta_2=\delta_1 e^{-k|\sigma|(a+|c|\delta_1)}$. The value of $\dUB$ from Example~\ref{Example:mu0}
is also shown in both plots.
} \label{Fig:Basinsmu0}
\end{center}
\end{figure}

\begin{example} \label{Example:mu0}
Consider \eqref{Eq:1Delay} with $c>0$, $a>0$, $\mu=0$ and $\sigma\in[-\pi/2a,-1/a)$ and
for $\delta\in[-1/(c\sigma),a/c)$
let $\varphi(t)$ be Lipschitz continuous with
\begin{gather*}
\varphi(0)=\delta, \qquad
\varphi(t)=-\delta\text{ for } t\leqslant -a-c\delta,\\
\varphi(t)\in(-\delta,\delta) \text{ for } t\in(-a-c\delta,0).  
\end{gather*}
Then while the deviated argument $\alpha(t,u(t))=t-a-cu(t)\leqslant -a-c\delta$ we have
\[u(t-a-cu(t))=-\delta \quad \text{and} \quad \dot u(t)=-\sigma\delta.\]
Hence,
\begin{equation}\label{Eq:Ex1}
u(t)=\delta(1-\sigma t)\geqslant \delta+\frac{t}{c} ,\quad \text{for }t\geqslant 0.
\end{equation}
But now $\alpha(t,u(t))\leqslant -a- c\delta$ for all $t\geqslant 0$ and \eqref{Eq:Ex1} is valid for all $t\geqslant 0$.
Thus for $\mu=0$, $\sigma\in[-\pi/2a,-1/a)$ on the axis between the third and fourth quadrants of the
stability region we have
$\|\varphi\|=\delta$ and $|u(t)|\to\infty$ as $t\to\infty$.
It follows that the steady state is not globally asymptotically stable and also that $B(0,\delta)$ is not
contained in its basin of attraction. Thus $\dUB=-1/c\sigma$ provides an upper bound on the radius of the largest ball contained in the basin of attraction.
\end{example}

Figure~\ref{Fig:Basinsmu0} shows three bounds on the basin of attraction of the steady state of \eqref{Eq:1Delay}.
The value of $\dUB$ from Example~\ref{Example:mu0} gives an upper bound on the radius of the largest ball contained in the
basin of attraction. Two lower bounds on the radius of the largest ball are also shown. The larger bound $\delta_1$
gives the radius of the ball that
Theorem~\ref{Thm:StabilityDDE:Razumikhin_k=123} shows
is contracted asymptotically to the steady state provided
the solution is sufficiently differentiable. Lemma~\ref{Lem:FiniteBoundI} is used to ensure that the solution
remains bounded long enough to acquire sufficient regularity, and the growth in the solution allowed by that lemma results
in the smaller radius $\delta_2$ (as defined by \eqref{Eq:basink})
of the ball that is contained in the basin of attraction for general continuous initial functions $\phi$.
We see that the bounds $\delta_1$ increase monotonically with $k$, but because of the exponential term in
\eqref{Eq:basink}, the largest value of $\delta_2$ is achieved with $k=1$ in most of the interval for
which $P(1,0,1)<1$.

Now consider the case of $\mu>0$. We can again derive an upper bound on the basin of attraction of the steady state
when $(\mu,\sigma)\in\cusp$.

\begin{figure}[t]
\begin{center}
\subfigure[{$\delta_1$ when $\sigma=-1$, $\mu\geq0$}]{\includegraphics[scale=1]{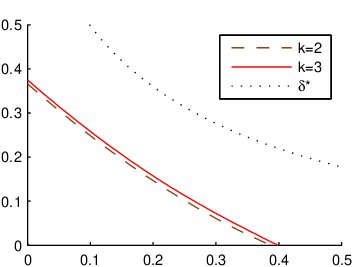}}
\put(-23,-2){{\footnotesize $\mu$}}
\put(-176,104){{\footnotesize $\delta$}}
\mbox{}\hspace{2em}\mbox{}
\subfigure[{$\delta_2$ when $\sigma=-1$, $\mu\geq0$}]{\includegraphics[scale=1]{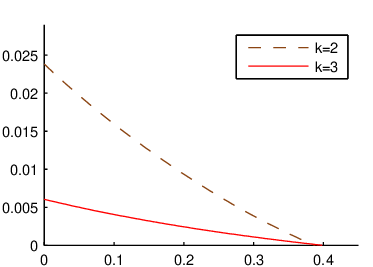}}
\put(-10,0){{\footnotesize $\mu$}}
\put(-174,108){{\footnotesize $\delta$}}
\caption{For fixed $a=c=1$, $\sigma=-1$ and $\mu>0$, (a) supremum of $\delta_1\in(0,a/c]$ such that $(\mu,\sigma)\in\{P(\delta_1,c,k)<\delta_1\}$
with $k=2$, $3$, and,
(b) $\delta_2=\delta_1 e^{-k(|\mu|+\sigma|)(a+|c|\delta_1)}$. The value of $\dUB$ from Example~\ref{Example:2}
is also shown in (a).
}  \label{Fig:Basinssigm1}
\end{center}
\end{figure}

\begin{example} \label{Example:2}
Let $a>0$, $c>0$, $\mu>0$ and $(\mu,\sigma)\in\cusp$ so $\sigma<-\mu<0$. Also let
$$q(\delta)=-a\mu-1-c\sigma\delta-\ln\bigl(c\delta(\mu-\sigma)\bigr).$$
Note that $q(1/(c(\mu-\sigma)))=-\mu(a+1/(\mu-\sigma))<0$,
while $q'(\delta)=-c\sigma-1/\delta<0$ for all $\delta\in(0,1/(c(\mu-\sigma)))$.
Also $q(\delta)\to\infty$ as $\delta\to 0$, hence there exists
$\dUB\in\bigl(0,1/(c(\mu-\sigma))\bigr)$ such that $q(\dUB)=0$ and $\dUB$ is unique in this interval.

Suppose that the parameters are chosen so that $\dUB<a/c$. A sufficient (but not necessary) condition for this is $\sigma<\mu-1/a$ since this implies $1/(c(\mu-\sigma))\leqslant a/c$.  Now let $\delta\in(\dUB,a/c)$ so $q(\delta)<0$ and consider \eqref{Eq:1Delay} with $\varphi(t)$
Lipschitz continuous and
\begin{gather*}
\varphi(0)=\delta, \qquad
\varphi(t)=-\delta\text{ for } t\leqslant \mu q(\delta),\\
\varphi(t)\in(-\delta,\delta) \text{ for } t\in(\mu q(\delta),0),
\end{gather*}
Then \eqref{Eq:1Delay} has solution
\begin{equation}
u(t)=\frac{\sigma \delta}{\mu} + \delta e^{\mu t} \left[\frac{\mu-\sigma}{\mu}\right]
\label{Eq:Ex2}
\end{equation}
with $u(t-a-cu(t))=-\delta$ for all $t\geqslant 0$. To see this note that
$$\alpha(t,u(t))=t-a-cu(t)=t-a-\frac{c\sigma\delta}{\mu}-c\delta e^{\mu t}\left[\frac{\mu-\sigma}{\mu}\right],$$
with $\alpha(t,u(t))\to-\infty$ as $t\to\infty$. Differentiating the expression for $\alpha(t,u(t))$ shows that
$\alpha(t,u(t))\leqslant \mu q(\delta)<0$ for all $t\geq0$,
with $\alpha(t,u(t))=\mu q(\delta)$ when $t=-\mu^{-1}\ln \left( c \delta (\mu-\sigma)\right)$.
Hence, as in Example~\ref{Example:mu0}, we have $\|\varphi\|=\delta$ and $|u(t)|\to\infty$ as $t\to\infty$. The steady state is not globally asymptotically stable and the ball $B(0,\delta)$ is not contained in its basin of attraction. Thus $\dUB$ provides an upper bound on the radius of the largest ball contained in the basin of attraction.
\end{example}

For $\sigma=-1$ and $\mu\geq0$, Figure~\ref{Fig:Basinssigm1} shows the same bounds $\delta_1$, $\delta_2$ and $\delta^*$ on the radius of the basin of attraction of the steady state as were shown in Figure~\ref{Fig:Basinsmu0}. Since these parameters are outside the set $\{P(1,0,1)<1\}$ no bound is shown for $k=1$. On nearly all of this interval $k=2$ gives the largest lower bound $\delta_2$ on the radius of
a ball contained in the basin of attraction.

\begin{figure}[t]
\begin{center}
\subfigure[$\delta_1$]{\includegraphics[scale=1]{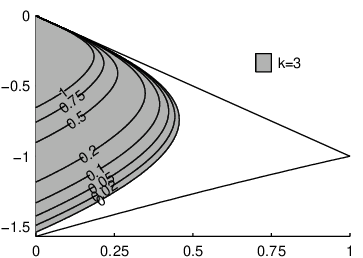}}
\put(-25,-2){{\footnotesize $\mu$}}
\put(-162,110){{\footnotesize $\sigma$}}
\subfigure[$\delta_2$]{\includegraphics[scale=1]{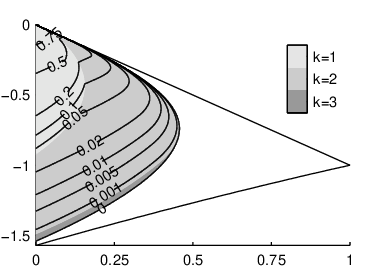}}
\put(-30,0){{\footnotesize $\mu$}}
\put(-170,100){{\footnotesize $\sigma$}}

\vspace{-2ex}

\subfigure[$\dUB$]{\includegraphics[scale=1]{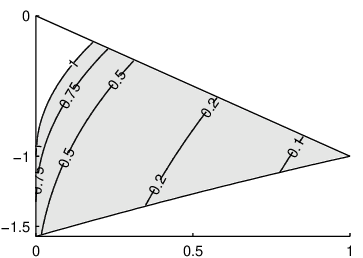}}
\put(-30,0){{\footnotesize $\mu$}}
\put(-170,90){{\footnotesize $\sigma$}}
\caption{Plot of $\cusp$ for fixed $a=c=1$ with contour plots of (a) the maximum $\delta_1\in(0,a/c]$
such that $(\mu,\sigma)\in\{P(\delta_1,c,k)<\delta_1\}$, and (b) the $\delta_2$ that maximizes $\delta_2=\delta e^{-k(|\mu|+|\sigma|)(a+|c|\delta}$ for $\delta\in(0,\delta_1]$. Shading shows the value $k\in\{1,2,3\}$ for
which the maximum is achieved. (c) The upper bound $\dUB$ from Example~\ref{Example:2} for the radius of the largest ball $B(0,\delta)$ contained in the basin of attraction of the zero solution to \eqref{Eq:1Delay}.}  \label{Fig:Basins2}
\end{center}
\end{figure}

Figure~\ref{Fig:Basins2} shows these bounds on the basin of attraction in the cusp $\cusp$.
The shaded region in Figure~\ref{Fig:Basins2}(c) denotes the portion of $\cusp$ for which $\dUB\leqslant a/c$ when $a=c=1$, and hence $\dUB$ from Example~\ref{Example:2} gives an upper bound on the radius of the largest ball contained in the basin of attraction. The corresponding bounds $\dUB$ are shown as contours within this
region. Figure~\ref{Fig:Basins2}(a) and (b) shows the lower bounds $\delta_1$ and $\delta_2$, along with the
value of $k$ that achieves the bound.

In all three figures in this section, $\delta_2$ is computed using \eqref{Eq:basink} and $\delta_1$
is obtained from solving $P(\delta_1,c,k)=\delta_1$ similarly to computations described in
Section~\ref{Sec:StabilityDDE:Measurements}.

%% file: Conclusions.tex
In this paper we have expanded upon the existing work on Lyapunov-Razumikhin techniques 
by providing results specifically tailored to DDEs with time-varying discrete delays including problems with
state-dependent delays and vanishing delays. Our main results provide sufficient conditions for Lyapunov and asymptotic stability of steady state solutions of DDEs in Theorems~\ref{Thm:Lyapunov_NA} and \ref{Thm:Technique} respectively. These conditions involve converting the DDE into a corresponding ODE problem with the delay terms treated as source terms that satisfy constraints. Our results require a Lipschitz condition on the right-hand side function $f$ in \eqref{Eq:GeneralMultipleDelay} instead of the more restrictive Lipschitz condition on $F$ in \eqref{Eq:RFDE} required in Barnea \cite{Barnea:1}, and do not require the construction of auxiliary functions as required by Hale and Verduyn Lunel~\cite{HaleLunel:1}. Nevertheless we are able to show asymptotic stability, using a proof
by contradiction showing that there cannot exist a solution which is not asymptotic to the steady state.

We apply our results to the model state-dependent DDE \eqref{Eq:1Delay} in
Sections~\ref{Sec:Stability:k=123}--\ref{Sec:Basins}.
The main result of the application of Lyapunov-Razumikhin techniques to \eqref{Eq:1Delay} is given as Theorems~\ref{Thm:StabilityDDE:Razumikhin_k=123}
where we prove that the zero solution to \eqref{Eq:1Delay} is asymptotically stable if $(\mu,\sigma)\in\{P(1,0,k)<1\}$, for $k=1$, $2$ or $3$ and provide lower bounds on the basin of attraction.

The parameter regions in which stability is proven in these theorems are compared in Section~\ref{Sec:StabilityDDE:Measurements}.
As shown in Figure~\ref{Fig:StabilityDDESummary}, the derived parameter regions grow as larger values of $k$ are used, 
though for $\mu\ne0$,
the derived stability region does not approach the entire known stability region $\Sigma_\star$ as $k\to \infty$
(for reasons discussed in
in Section~\ref{Sec:StabilityDDE:Measurements}).

In Section~\ref{Sec:Basins} we consider \eqref{Eq:1Delay} in the cusp $\cusp$
where $\mu>0$ and the steady state would be unstable without the delay term.
In Examples~\ref{Example:mu0} and~\ref{Example:2} we constructed solutions
which do not converge to the steady state for $(\mu,\sigma)\in\cusp$.
These solutions provide us with an upper bound $\dUB$ on the radius of the largest ball about the zero solution
contained in the basin of attraction.
In Figures~\ref{Fig:Basinsmu0}--\ref{Fig:Basins2} these upper bounds were compared with the lower bound $\delta_2$ on the basin of attraction from \eqref{Eq:basink}.



In the current work have studied stability through Lyapunov-Razumikhin techniques, but let us briefly compare and contrast this approach to the alternative, namely linearization. State-dependent DDEs
have long been linearized by freezing the delays at their steady-state values and linearizing the resulting constant delay DDE \cite{Cooke,CookeHuang}. This heuristic approach has recently been put on a rigorous footing.
For a class of state-dependent DDEs which includes \eqref{Eq:1Delay} with $\mu=0$,
Gy\"ori and Hartung \cite{GyoriHartung:2} proved that the steady state of the state-dependent DDE is
exponentially stable if and only if the steady state of the corresponding frozen-delay DDE is
exponentially stable. In \cite{GyoriHartung:1} they generalise this result to a class of nonautonomous problems
which are linear except for the state-dependency.

To compare and contrast our results with the linearization results of \cite{GyoriHartung:1}, we note
that our results apply to
a larger class of problems \eqref{Eq:GeneralMultipleDelay} than was considered in \cite{GyoriHartung:1},
and we prove both Lyapunov stability and asymptotic stability results, whereas \cite{GyoriHartung:1} is concerned
with exponential stability. The results in \cite{GyoriHartung:1} do apply directly to our
model problem \eqref{Eq:1Delay}, and
reveal the parameter region for which the steady state is exponentially stable. In contrast our Lyapunov-Razumikhin
techniques are only able to deduce stability in part of this parameter region.



Even though Lyapunov-Razumikhin techniques do not provide a proof of stability in the entire known stability region for \eqref{Eq:1Delay}, just as Lyapunov functions for ODEs do not always do so, they can nevertheless still be
a very useful tool for studying stability in state-dependent DDEs. In particular our
Lyapunov stability result is applicable to nonautonomous problems (for some of which rigorous linearization
has yet to be derived) and the asymptotic stability result
yield bounds on the basins of attraction which cannot be derived through linearization.


%% file: AppendixExplicit.tex
Here we prove Theorem~\ref{Thm:StabilityDDE:simplify}.
Let $\mathcal{I}_{\!1}$ and $\mathcal{I}_{\!2}$ be defined by \eqref{Eq:I1} and \eqref{Eq:I2}.
Recall that $\mathcal{I}_{\!1}$ only applies when $\frac{\delta  + \hat u}{D_1\delta}\geq r$,
in which case the integration does not have to be split into two intervals. For this case to occur we require
$\hat u\in[(r D_1-1)\delta,-\mu\delta/\sigma]$, which
is only possible in the region where $r D_1-1 \leq -\mu/\sigma$.
When $\frac{\delta  + \hat u}{D_1\delta}<r$ the integration has to be broken into two parts
and $\mathcal{I}_{\!2}$ applies. In that case
$\hat u \in \bigl[-\delta,\min \bigl\{(rD_1-1)\delta,-\mu\delta/\sigma\bigr\}\bigr]$.
We require the following lemmas.

\begin{lem}\label{Lem:StabilityDDE:tauD-1<-mu/sigma}
Let $a>0$, $c\ne 0$, $\sigma\leqslant \mu$ and $\sigma<-\mu$. Let $\delta\in(0,|a/c|)$ and $rD_1-1\leqslant -\mu/\sigma$.
If $\mu\geqslant 0$ then $\sigma\geqslant-1/r$.
If $\mu<0$ then $\mu\in\bigl[(-3+2\sqrt{2})/r,0\bigr]$ and
\begin{equation} \label{eq:lemmaB1}
\sigma \geqslant  - \tfrac{1}{r}\biggl[ \tfrac12(1 + \mu r)
+\tfrac{1}{2}\sqrt{1 + 6\mu r+(\mu r)^2}\biggr]\geqslant -\tfrac{1}{r}.
\end{equation}
\end{lem}

\begin{proof}
Let $r D_1-1 \leq -\mu/\sigma$. Then
\begin{gather} \notag
r(\text{sign}(\mu)\mu-\sigma)- 1 = r(|\mu|+|\sigma|)- 1 \leqslant rD_1 - 1 \leqslant  - \tfrac{\mu}{\sigma},\\
 \Rightarrow \quad
r\sigma ^2  + (1-\text{sign}(\mu)\mu r)\sigma  - \mu  \leqslant 0.
\label{Eq:tauD-1<-mu/sigma1}
\end{gather}
The boundary of the region where this inequality holds is
\begin{equation}
\sigma  =  - \tfrac{1}{r}\Bigl[
\tfrac{1}{2}({1-\text{sign}(\mu)\mu r}) \pm \tfrac{1}{2}\sqrt{(1-\text{sign}(\mu)\mu r)^2  + 4\mu r}
\Bigr].
\label{Eq:tauD-1<-mu/sigma2}
\end{equation}
If $\mu\geqslant 0$ then this simplifies to $\sigma=-1/r$ and $\sigma=\mu$. Since $\sigma\leqslant 0$, and also $\mu=\sigma=0$ satisfies \eqref{Eq:tauD-1<-mu/sigma1} it follows that \eqref{eq:lemmaB1} holds
for $\sigma\in[-1/r,0]$ for the case $\mu\geqslant 0$.

If $\mu<0$ then \eqref{Eq:tauD-1<-mu/sigma2} simplifies to
\begin{equation}
\sigma  =
- \tfrac{1}{r}\Bigl[ {\tfrac{1}{2}(1 + \mu r)\pm\tfrac{1}{2}\sqrt{1 + 6\mu r +(\mu r)^2 } } \Bigr].
\label{Eq:tauD-1<-mu/sigma3}
\end{equation}
Requiring $1+6\mu r+(\mu r)^2\geqslant 0$ yields $\mu r\in[-3+2\sqrt{2},0]$. The lower bound on $\sigma$ can be found by taking the lower boundary in \eqref{Eq:tauD-1<-mu/sigma3} which attains its minimum at $\mu=0$. This yields
\eqref{eq:lemmaB1}.
\end{proof}

\begin{lem} \label{Lem:StabilityDDE:simplify}
Let $a>0$, $\sigma\leqslant \mu$ and $\sigma<-\mu$. Let $\delta\in(0,|a/c|)$. Define
\[
\hat u_* = \left[ \tfrac{1}{\mu}D_1\ln \left( 1 - \tfrac{\mu}{\sigma}e^{\mu r} \right) - 1 \right]\delta.
\]
The following statements are true:
\begin{enumerate}[(A)]
\item
If $rD_1 - 1 >  -\mu/\sigma$
then the maximum of $\mathcal {I}_{\!2}(\hat u ,\delta)$ over $\hat u \in[-\delta,-\mu\delta/\sigma]$
occurs at $\hat u _*$ if $\hat u_*< -\mu\delta/\sigma$, and at $-\mu\delta/\sigma$ otherwise.
\item
If $rD_1 - 1 > -\mu/\sigma$ and $u_*< -\mu\delta/\sigma$ then $P(\delta,c,2)\geqslant \delta$.
\item
If $rD_1 - 1 \leqslant -\mu/\sigma$ then
$\sup_ {\hat u \in [-\delta,(rD_1-1)\delta]}\mathcal{I}_{\!2}(\hat u ,\delta)
=\mathcal{I}_{\!2}((rD_1-1)\delta,\delta)$
\item
If $rD_1 - 1 \leqslant -\mu/\sigma$ then
$\sup_ {\hat u \in [(rD_1-1)\delta,-\mu\delta/\sigma]}\mathcal{I}_{\!1}(\hat u ,\delta)
=\mathcal{I}_{\!1}(-\mu\delta/\sigma,\delta)$
\item
If $rD_1 - 1 \leqslant -\mu/\sigma$ then
$\mathcal{I}_{\!2}((rD_1-1)\delta,\delta)\leqslant \mathcal{I}_{\!1}(-\mu\delta/\sigma,\delta)$.
\end{enumerate}
\end{lem}

\begin{proof}[Proof of (A)]
To find $\hat u _*$ which maximises $\mathcal{I}_{\!2}(\hat u_*,\delta)$
consider the derivative
\begin{equation} \label{Eq:I2derivative}
\tfrac{\partial}{\partial \hat u}\mathcal{I}_{\!2}(\hat u,\delta) =
e^{\mu r}  + \tfrac{\sigma}{\mu} \Bigl( e^{\mu\tfrac{\delta  + \hat u}{D_1\delta}} - 1\Bigr).
\end{equation}
At $\hat u = -\delta$ this is positive. Setting the derivative equal to zero in \eqref{Eq:I2derivative} yields
$$\hat u_* = \left[ {\tfrac{1}{\mu}D_1\ln\left( 1 - \tfrac{\mu}{\sigma}e^{\mu r}  \right) - 1} \right]\delta.$$
Since ${1 - \tfrac{\mu } {\sigma }e^{\mu r} } \in \left[0,1\right]$ if $\mu<0$ and ${1 - \tfrac{\mu } {\sigma }e^{\mu r} } >1$ if $\mu>0$ then $\hat u_*>-\delta$ in both cases.
\end{proof}

\begin{proof}[Proof of (B)]
First we need to prove that if $\hat u_*<-\mu\delta/\sigma$ and $\mu>0$ then $rD_1 - 1 \leqslant -\mu/\sigma$.
Let $\hat u_*<-\mu\delta/\sigma$ and $\mu>0$. Then
$\frac{\partial}{\partial \hat u}\mathcal{I}_{\!2}(-\mu\delta/\sigma,\delta) <0$.
Consider the term $D_1/\mu$,
\begin{equation}
\frac{D_1}{\mu}
= \frac{|\mu| + |\sigma|}{\mu}\bigl( 1 + (|\mu| + |\sigma|)|c|\delta\bigr)
\geqslant \frac{|\mu| + |\sigma|}{\mu}
= \left( 1 - \frac{\sigma}{\mu} \right).
\label{Eq:epsD/mu}
\end{equation}
Now consider the exponent of the second term in \eqref{Eq:I2derivative} with $\hat u=-\mu\delta/\sigma$,
\[
\mu\frac{1 - \mu/\sigma}{D_1}
= \left( 1 - \frac{\mu}{\sigma} \right)\frac{\mu}{D_1}
\leqslant \frac{1 - \mu/\sigma}{1 -\sigma/\mu}
=  - \frac{\mu}{\sigma}.
\]
Thus,
\[
e^{\mu r}  + \tfrac{\sigma}{\mu}(e^{-\mu/\sigma} - 1)
\leqslant
\tfrac{\partial}{\partial \hat u}\mathcal{I}_{\!2}(-\mu\delta/\sigma,\delta) <0.
\]
Isolating $r$ in this expression yields
$r<\tfrac{1}{\mu}\ln \bigl[ \tfrac{\sigma}{\mu}\bigl( 1 - e^{-\mu/\sigma}\bigr) \bigr] $.
Let $x=\mu/\sigma$, then $x \in (-1,0)$
and $(1-e^{-x})/x>1$.
Also, $(1-1/x)\ln[(1-e^{-x})/x]-1\leqslant -x$. These inequalities and \eqref{Eq:epsD/mu} imply
\[
r D_1 - 1 \leqslant
\tfrac{D_1}{\mu}\ln \left[ \tfrac{\sigma}{\mu}(1 - e^{-\mu/\sigma}) \right] - 1
\leqslant \left( 1 - \tfrac{\sigma}{\mu} \right)\ln\left[\tfrac{\sigma}{\mu}(1 - e^{-\mu/\sigma}) \right] - 1
\leqslant  - \tfrac{\mu}{\sigma}.
\]
Thus if $r D_1 - 1 > -\mu/\sigma$ then $\hat u_*<-\mu\delta/\sigma$ can only occur if $\mu<0$.

Now let $\mu<0$ and $\hat u_*<-\mu\delta/\sigma$. Then by setting $ \tfrac{{\partial }}{{\partial \hat u}} \mathcal{I}_{\!2}(\hat u_*,\delta)=0$ in \eqref{Eq:I2derivative},
$e^{\mu \tfrac{\delta  + \hat u_*}{D_1\delta }}  - 1= -\tfrac{\mu}{\sigma} e^{\mu r}  $.
Also, $e^{\mu r}-\sigma/\mu<0$ because $\mu<0$ and $\sigma\leqslant \mu<0$. Thus,
\begin{align*}
\mathcal{I}_{\!2}(\hat u_* ,\delta) & =
\hat u _*\left(e^{\mu r}-\tfrac{\sigma}{\mu}\right) + \tfrac{\sigma}{\mu}\delta
\biggl[\tfrac{D_1}{\mu}\Bigl( e^{\mu\tfrac{\delta  + \hat u_*}{D_1\delta }}  - 1 \Bigr) - e^{\mu r}\biggr],\\
& \geqslant -\tfrac{\mu}{\sigma} \delta\left(e^{\mu r}-\tfrac{\sigma}{\mu}\right)
+ \tfrac{\sigma}{\mu}\delta \left[ -\tfrac{D_1}{\sigma}e^{\mu r}- e^{\mu r} \right]
=\delta-\left(\tfrac{D_1}{\mu } + \tfrac{\mu}{\sigma} + \tfrac{\sigma}{\mu}\right)\delta e^{\mu r}.
\end{align*}
But $D_1 \geq |\mu| +|\sigma| \geqslant |\mu/\sigma| |\mu| + |\sigma|
= - (\mu^2/\sigma + \sigma)$ which implies
$\tfrac{D_1}{\mu} + \tfrac{\mu}{\sigma} + \tfrac{\sigma}{\mu} \leqslant 0$.
Hence $\mathcal{I}_{\!2}(\hat u_* ,\delta)\geq\delta$ as required.
\end{proof}

\begin{proof}[Proof of (C)]
Let $\frac{\partial}{\partial \hat u}\mathcal{I}_{\!2}((rD_1-1)\delta,\delta)<0$. Then,
$e^{\mu r}  + \tfrac{\sigma}{\mu}( e^{\mu r}  - 1 ) < 0$,
which can be rewritten as
\begin{equation} \label{Eq:FullIntegration,Expr1<0sigma}
\sigma  <\frac{{\mu e^{\mu r} }}
{1 - e^{\mu r} }=-\frac{1}{r}\left(\frac{-\mu r e^{\mu r} }
{1 - e^{\mu r} }\right).
\end{equation}
We show that the expression on the right-hand-side is continuous and decreases as $\mu$ increases.
Let $x=\mu r$, then
\begin{gather*}
\mathop{\lim}\limits_{\mu\to 0} -\frac{1}{r}\left(\frac{-\mu r e^{\mu r}}{1 - e^{\mu r} }\right)
= \mathop{\lim}\limits_{x\to 0} -\frac{1}{r}\left(\frac{-x e^{x}}{1 - e^{x}}\right)=  - \frac{1}{r},\\
\frac{d}
{{d\mu }}\left[-\frac{1}{r}\left(\frac{-\mu r e^{\mu r} }
{1 - e^{\mu r} }\right)\right]=\frac{d}
{{dx }}\frac{x e^{x} }
{1 - e^{x} } =\frac{e^x\left(1+x-e^x\right)}{\left(1-e^x\right)^2}\leqslant 0.
\end{gather*}
So when $\mu>0$, a necessary condition for
\eqref{Eq:FullIntegration,Expr1<0sigma}
to hold is $\sigma  < - 1/r$.
From Lemma~\ref{Lem:StabilityDDE:tauD-1<-mu/sigma}, if $rD_1-1 \leqslant -\mu/\sigma$ and $\mu>0$
then $\sigma\geqslant -1/r$. Thus
$ \tfrac{\partial  }{{\partial \hat u}}\mathcal{I}_{\!2}((rD_1-1)\delta,\delta)\geqslant 0$ if $rD_1-1 \leqslant -\mu/\sigma$ and $\mu>0$.

Now let $rD_1-1 \leqslant -\mu/\sigma$ and $\mu<0$. From Lemma~\ref{Lem:StabilityDDE:tauD-1<-mu/sigma},
$\mu r\in[-3+2\sqrt{2},0]$ and
$\sigma \geqslant  - \frac{1}{{r }}\left[ {\frac{1}{2}\left( {1 + \mu r} \right)- \frac{1}
{2}\sqrt {1 + 6\mu r+ (\mu r)^2 } } \right]$.
Since $\tfrac{-x e^{x}}{1 - e^{x}} \geqslant \tfrac{1}
{2}\left( {1 + x} \right)+ \tfrac{1}
{2}\sqrt {1 + 6x + x^2 }  $ for $x\in[-3+2\sqrt{2},0]$, then
$\sigma \geqslant -\tfrac{1}{r}\left(\tfrac{-\mu r e^{\mu r} }
{1 - e^{\mu r} }\right)$. This contradicts \eqref{Eq:FullIntegration,Expr1<0sigma}. Thus $ \tfrac{{\partial  }}
{{\partial \hat u}}\mathcal{I}_{\!2}((rD_1-1)\delta,\delta)\geqslant 0$ if $rD_1-1 \leqslant -\mu/\sigma$
and $\mu<0$. Thus,
\begin{equation} \label{Eq:tauD-1<-mu/sigma-->der>0}
rD_1-1 \leqslant -\tfrac{\mu}{\sigma} \quad \Rightarrow \quad
\tfrac{\partial}{\partial \hat u}\mathcal{I}_{\!2}((rD_1-1)\delta,\delta)
=e^{\mu r}  + \tfrac{\sigma}{\mu}\left( {e^{\mu r}  - 1} \right) \geqslant 0.
\end{equation}
This is shown 
in Figure~\ref{Fig:tauD-1<-mu/sigma-->der>0}. To finish 
the proof of (C), observe from \eqref{Eq:I2derivative} that 
$\tfrac{\partial}{\partial \hat u}\mathcal{I}_{\!2}(\hat u,\delta)$
decreases as $\hat u$ increases. Then by \eqref{Eq:tauD-1<-mu/sigma-->der>0},
$\sup_{\hat u \in[-\delta,(rD_1-1)\delta]}\mathcal{I}_{\!2}(\hat u ,\delta)=\mathcal{I}_{\!2}((rD_1-1)\delta,\delta)$.
\end{proof}

\begin{figure}[t]
\begin{center}
\includegraphics[scale=1]{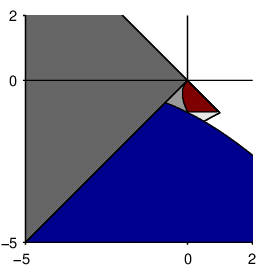}
\put(-23,-2){{\footnotesize $\mu$}}
\put(-130,102){{\footnotesize $\sigma$}}
\FigureCaption{\label{Fig:tauD-1<-mu/sigma-->der>0} Region where $rD_1-1\leqslant -\mu/\sigma$ is shown in brown and the region where $e^{\mu r}  + (\sigma/\mu)( e^{\mu r}  - 1 ) <0 $ is shown in blue.
These two regions do not intersect.}
\end{center}
\end{figure}

\begin{proof}[Proof of (D)]
Let $rD_1-1 \leqslant -\mu/\sigma$. For all $\hat u \in \bigl[(rD_1-1)\delta,-\mu\delta/\sigma\bigr]$,
$$\tfrac{\partial}{\partial \hat u}\mathcal{I}_{\!1}(\hat u ,\delta)
= e^{\mu r}  + \tfrac{\sigma}{\mu}(e^{\mu r}  - 1)
=\tfrac{\partial}{\partial \hat u}\mathcal{I}_{\!2}((rD_1-1)\delta,\delta).$$
From \eqref{Eq:tauD-1<-mu/sigma-->der>0},
$\frac{\partial}{\partial \hat u}\mathcal{I}_{\!1}(\hat u ,\delta)>0$
for all $\hat u\in [-\delta,(rD_1-1)\delta]$. Thus,
$\sup_{\hat u\in [(rD_1-1)\delta, -\mu\delta/\sigma]}\mathcal{I}_{\!1}(\hat u,\delta)
= \mathcal{I}_{\!1}(-\mu\delta/\sigma,\delta)$.
\end{proof}

\begin{proof}[Proof of (E)]
From \eqref{Eq:I1} and \eqref{Eq:I2} we find that
$$
\mathcal{I}_{\!1}(-\mu\delta/\sigma,\delta)-\mathcal{I}_{\!2}((rD_1 - 1)\delta ,\delta)
= \left(  - \tfrac{\mu}{\sigma} - ( rD_1 - 1) \right)\delta \left[ e^{\mu r}  + \tfrac{\sigma}{\mu}( e^{\mu r}  - 1) \right]. $$
For $rD_1-1 \leqslant -\mu/\sigma$ it follows from \eqref{Eq:tauD-1<-mu/sigma-->der>0}
that this expression is non-negative. Thus,
$\mathcal{I}_{\!1}(-\mu\delta/\sigma,\delta) \geqslant \mathcal{I}_{\!2}((rD_1 - 1)\delta,\delta)$.
\end{proof}

\noindent
Finally we can prove Theorem~\ref{Thm:StabilityDDE:simplify}.


\begin{proof}[Proof of Theorem~\ref{Thm:StabilityDDE:simplify}]
Note that this expression does not hold outside of $\bigl\{P(\delta,c,2)<\delta\bigr\}$.
In order to prove this theorem, we need items (A)-(E) in Lemma~\ref{Lem:StabilityDDE:simplify} which require Lemma~\ref{Lem:StabilityDDE:tauD-1<-mu/sigma}.

Let $r D_1-1>-\mu/\sigma$. Then we can only have the two-part integration so
$\mathcal{I}(\hat u,\delta,c,2)=\mathcal{I}_{\!2}(\hat u,\delta)$. From (A) and (B), $P(\delta,c,2)=\mathcal{I}_{\!}\bigl(-\mu/\sigma,\delta\bigr)$ if $P(\delta,c,2)<\delta$.

Let $r D_1-1\leq-\mu/\sigma$. Then we can have either the one-part or the two-part integration. From (C) and (D), $\mathcal{I}(\hat u,\delta,c,2)
=\max\bigl\{\mathcal{I}_{\!2}((rD_1-1)\delta,\delta),\mathcal{I}_{\!1}(-\mu\delta/\sigma,\delta)\bigr\}$.
From (E), $P(\delta,c,2)=\mathcal{I}_{\!1}(-\mu\delta/\sigma,\delta\bigr)$.
\end{proof}

%% file: RazumikhinBoundsARXIV4.bbl
\begin{thebibliography}{10}

\bibitem{Barnea:1}
D.I. Barnea.
\newblock A method and new results for stability and instability of autonomous
  functional differential equations.
\newblock {\em SIAM J. Appl. Math}, 17:681--697, 1969.

\bibitem{BellenZennaro:1}
A.~Bellen and M.~Zennaro.
\newblock {\em Numerical Methods for Delay Differential Equations}.
\newblock Numerical Mathematics and Scientific Computation. Oxford Science
  Publications, New York, 2003.

\bibitem{CHK:1}
R.C. Calleja, A.R. Humphries, and B.~Krauskopf.
\newblock Resonance phenomena in a scalar delay differential equation with two
  state-dependent delays.
\newblock {\em SIAM J. Appl. Dyn. Syst.}, 16:1474--1513, 2017.

\bibitem{Cooke}
K.L. Cooke.
\newblock Asymptotic theory for the delay-differential equation {{\(u'(t) =
  -au(t - r(u(t)))\)}}.
\newblock {\em J. Math. Anal. Appl.}, 19:160--173, 1967.

\bibitem{CookeHuang}
K.L. Cooke and W.~Huang.
\newblock On the problem of linearization for state-dependent delay
  differential equations.
\newblock {\em Proc. Amer. Math. Soc.}, 124:1417--1426, 1996.

\bibitem{CHM16}
M.~Craig, A.R. Humphries, and M.C. Mackey.
\newblock A mathematical model of granulopoiesis incorporating the negative
  feedback dynamics and kinetics of {G-CSF}/neutrophil binding and
  internalization.
\newblock {\em Bull. Math. Biol.}, 78:2304--2357, 2016.

\bibitem{DGVLW95}
O.~Diekmann, S.A. van Gils, S.M. Verduyn~Lunel, and H.-O. Walther.
\newblock {\em Delay Equations Functional-, Complex-, and Nonlinear Analysis},
  volume 110 of {\em Applied Mathematical Sciences}.
\newblock Springer-Verlag, 1995.

\bibitem{Driver:1}
R.D. Driver.
\newblock Existence theory for a delay-differential system.
\newblock {\em Contribs Diff. Eqns.}, 1:317--336, 1963.

\bibitem{ElsgoltsNorkin:1}
L.E. El'sgol'ts and S.B. Norkin.
\newblock {\em Introduction to the Theory and Application of Differential
  Equations with Deviating Arguments}.
\newblock Academic Press, New York, 1973.

\bibitem{GW13}
S.~Guo and J.~Wu.
\newblock {\em Bifurcation Theory of Functional Differential Equations}, volume
  184 of {\em Applied Mathematical Sciences}.
\newblock Springer-Verlag, 2013.

\bibitem{GyoriHartung:2}
I.~Gy\"ori and F.~Hartung.
\newblock On the exponential stability of a state-dependent delay equation.
\newblock {\em Acta Sci. Math. (Szeged)}, 66:71--84, 2000.

\bibitem{GyoriHartung:1}
I.~Gy\"ori and F.~Hartung.
\newblock Exponential stability of a state-dependent delay system.
\newblock {\em Discrete Contin. Dyn. Syst. Ser. A}, 18:773--791, 2007.

\bibitem{Hale:1}
J.K. Hale.
\newblock {\em Theory of Functional Differential Equations}, volume~3 of {\em
  Applied Mathematical Sciences}.
\newblock Springer-Verlag, New York, 1977.

\bibitem{HaleLunel:1}
J.K. Hale and S.M. Verduyn~Lunel.
\newblock {\em Introduction to Functional Differential Equations}, volume~99 of
  {\em Applied Mathematical Sciences}.
\newblock Springer-Verlag, New York, 1993.

\bibitem{HartungKrisztinWaltherWu:1}
F.~Hartung, T.~Krisztin, H.-O. Walther, and J.~Wu.
\newblock Functional differential equations with state-dependent delays: theory
  and applications.
\newblock In A.~Ca\~nada, P.~Dr\'abek, and A.~Fonda, editors, {\em Handbook of
  Differential Equations: Ordinary Differential Equations}, volume~3, pages
  435--545. Elsevier/North Holland, 2006.

\bibitem{Hayes50}
N.D. Hayes.
\newblock Roots of the transcendental equation associated with a certain
  difference-differential equation.
\newblock {\em J. London Math. Soc.}, s1-25:226--232, 1950.

\bibitem{HBCHM:1}
A.R. Humphries, D.A. Bernucci, R.~Calleja, N.~Homayounfar, and M.~Snarski.
\newblock Periodic solutions of a singularly perturbed delay differential
  equation with two state-dependent delays.
\newblock {\em J. Dyn. Diff. Eqs.}, 28:1215--1263, 2016.

\bibitem{Humphries:1}
A.R. Humphries, O.~DeMasi, F.M.G. Magpantay, and F.~Upham.
\newblock Dynamics of a delay differential equation with multiple state
  dependent delays.
\newblock {\em Discrete Contin. Dyn. Syst. Ser. A}, 32:2701--2727, 2012.

\bibitem{IS11}
T.~Insperger and St\'ep\'an.
\newblock {\em Semi-Discretization for Time-Delay Systems}, volume 178 of {\em
  Applied Mathematical Sciences}.
\newblock Springer-Verlag, 2011.

\bibitem{IST07}
T.~Insperger, G.~St\'ep\'an, and J.~Turi.
\newblock State-dependent delay in regenerative turning processes.
\newblock {\em Nonlinear Dyn.}, 47:275--283, 2007.

\bibitem{IvanovLizTrofimchuk:1}
A.~Ivanov, E.~Liz, and S.~Trofimchuk.
\newblock Halanay inequality, {Y}orke 3/2 stability criterion, and differential
  equations with maxima.
\newblock {\em Tohoku Math. J.}, 54:277--295, 2002.

\bibitem{Kato:1}
J.~Kato.
\newblock On {Liapunov-Razumikhin} type theorems for functional differential
  equations.
\newblock {\em Funkcialaj Ekvacioj}, 16:225--239, 1973.

\bibitem{KE:1}
G.~Kozyreff and T.~Erneux.
\newblock Singular {Hopf} bifurcation in a differential equation with large
  state-dependent delay.
\newblock {\em Proc. Roy. Soc. A}, 470:0596, 2013.

\bibitem{Krasovskii:1}
N.N. Krasovskii.
\newblock {\em Stability of Motion Applications of Lyapunov's Second Method to
  Differential Equations with Delay}.
\newblock Stanford University Press, Stanford, California, 1963.

\bibitem{Krisztin:1}
T.~Krisztin.
\newblock Stability for functional differential equations and some variational
  problems.
\newblock {\em Tohoku Math. J.}, 42:402--417, 1990.

\bibitem{Krisztin:3}
T.~Krisztin.
\newblock On stability properties for one-dimensional functional differential
  equations.
\newblock {\em Funkcialaj Ekvacioj}, 34:241--256, 1991.

\bibitem{Liz2003309}
E.~Liz, M.~Pinto, G.~Robledo, S.~Trofimchuk, and V.~Tkachenko.
\newblock Wright type delay differential equations with negative {Schwarzian}.
\newblock {\em Discrete and Contin. Dyn. Syst. Ser. A}, 9:309--321, 2003.

\bibitem{Mackey89}
M.C. Mackey.
\newblock Commodity price fluctuations: Price dependent delays and
  nonlinearities as explanatory factors.
\newblock {\em J. Econ. Theory}, 48:497 -- 509, 1989.

\bibitem{Magpantay:1}
F.M.G. Magpantay.
\newblock {\em On the stability and numerical stability of a model state
  dependent delay differential equation}.
\newblock PhD thesis, McGill University, Department of Mathematics and
  Statistics, 2012.

\bibitem{MH:1}
F.M.G. Magpantay and A.R. Humphries.
\newblock Generalised {Lyapunov-Razumikhin} techniques for scalar
  state-dependent delay differential equations.
\newblock {\em Discrete Contin. Dyn. Syst. Ser. S}, 13:85--104, 2020.

\bibitem{MalletParetNussbaum:1}
J.~Mallet-Paret and R.D. Nussbaum.
\newblock Boundary layer phenomena for differential-delay equations with
  state-dependent time lags, {I}.
\newblock {\em Arch. Ration. Mech. Anal.}, 120:99--146, 1992.

\bibitem{MalletParetNussbaum:2}
J.~Mallet-Paret and R.D. Nussbaum.
\newblock Boundary layer phenomena for differential-delay equations with
  state-dependent time lags: {II}.
\newblock {\em J. Reine Angew. Math.}, 477:129--197, 1996.

\bibitem{MalletParetNussbaum:3}
J.~Mallet-Paret and R.D. Nussbaum.
\newblock Boundary layer phenomena for differential-delay equations with
  state-dependent time lags: {III}.
\newblock {\em Discrete Contin. Dyn. Syst. Ser. A}, 189:640--692, 2003.

\bibitem{MalletParetNussbaum:4}
J.~Mallet-Paret and R.D. Nussbaum.
\newblock Superstability and rigorous asymptotics in singularly perturbed
  state-dependent delay-differential equations.
\newblock {\em J. Diff. Eqns.}, 250:4037--4084, 2011.

\bibitem{Matlab}
Mathworks.
\newblock {\em MATLAB 2017a}.
\newblock Mathworks, Natick, Massachusetts, 2017.

\bibitem{Myshkis:1}
A.~Myshkis.
\newblock Razumikhin's method in the qualitative theory of processes with
  delay.
\newblock {\em J. Appl. Math. Stoch. Anal.}, 8:233--247, 1995.

\bibitem{Razumikhin:1}
B.S. Razumikhin.
\newblock An application of {Lyapunov} method to a problem on the stability of
  systems with a lag.
\newblock {\em Autom. Remote Control}, 21:740--748, 1960.

\bibitem{Smith:1}
H.~Smith.
\newblock {\em An Introduction to Delay Differential Equations with
  Applications to the Life Sciences}.
\newblock Texts in Applied Mathematics. Springer, New York, 2011.

\bibitem{Stumpf:1}
E.~Stumpf.
\newblock Local stability analysis of differential equations with
  state-dependent delay.
\newblock {\em Discrete Contin. Dyn. Syst. Ser. A}, 36:3445--3461, 2016.

\bibitem{W03}
H.-O. Walther.
\newblock On a model for soft landing with state dependent delay.
\newblock {\em J. Dyn. Diff. Eqns.}, 19:593--622, 2003.

\bibitem{Winston74}
E.~Winston.
\newblock Uniqueness of solutions of state dependent delay differential
  equations.
\newblock {\em Journal of Mathematical Analysis and Applications}, 47:620 --
  625, 1974.

\bibitem{Wright:1}
E.M. Wright.
\newblock A non-linear difference-differential equation.
\newblock {\em J. Reine Angew. Math.}, 194:66--87, 1955.

\bibitem{Yorke:1}
J.A. Yorke.
\newblock Asymptotic stability for one dimensional differential-delay
  equations.
\newblock {\em J. Diff. Eqns.}, 7:189--202, 1970.

\end{thebibliography}
